\documentclass{amsart}
\usepackage{amssymb,xspace}
\usepackage{tikz-cd}
\usepackage{a4wide}
\usepackage{bbm} 
\usepackage{enumitem}
\setlist[enumerate]{label=\roman*)}
\theoremstyle{plain}
\newtheorem{nthr}{Theorem}[section]
\newtheorem{theorem}		[nthr]{Theorem}
\newtheorem{proposition}[nthr]{Proposition}
\newtheorem{lemma}		[nthr]{Lemma}
\newtheorem{corollary}	[nthr]{Corollary}

\newtheorem{remark}		[nthr]{Remark}
\theoremstyle{definition}

\def\mat#1{\ensuremath{#1}\xspace}
\def\dmat#1#2{\gdef#1{\mat{#2}}}
\def\oper#1#2{\dmat#1{\operatorname{#2}}}

\dmat\bk{\Bbbk}
\dmat\bA{\mathbb{A}}
\dmat\bC{\mathbb{C}}
\dmat\bF{\mathbb{F}}
\dmat\bG{\mathbb{G}}
\dmat\bH{\mathbb{H}}
\dmat\bL{\mathbb{L}}
\dmat\bN{\mathbb{N}}
\dmat\bQ{\mathbb{Q}}
\dmat\bP{\mathbb{P}}
\dmat\bR{\mathbb{R}}
\dmat\bT{\mathbb{T}}
\dmat\bZ{\mathbb{Z}}
\dmat\cA{\mathcal{A}}
\dmat\cC{\mathcal{C}}
\dmat\cD{\mathcal{D}}
\dmat\cF{\mathcal{F}}
\dmat\cE{\mathcal{E}}
\dmat\cH{\mathcal{H}}
\dmat\cM{\mathcal{M}}
\dmat\cO{\mathcal{O}}
\dmat\cP{\mathcal{P}}
\dmat\cS{\mathcal{S}}
\dmat\cT{\mathcal{T}}
\dmat\cU{\mathcal{U}}
\dmat\cV{\mathcal{V}}
\dmat\cZ{\mathcal{Z}}
\dmat\sA{\mathsf{A}}
\dmat\sH{\mathsf{H}}
\dmat\sI{\mathsf{I}}
\dmat\sJ{\mathsf{J}}
\dmat{\bfM}{\mathbf{M}}

\dmat\al\alpha
\dmat\be\beta
\dmat\ga\gamma
\dmat\de\delta
\dmat\eps{\varepsilon}
\dmat\hi\chi
\dmat\vi\varphi
\dmat\te\theta
\dmat\ka\kappa
\dmat\si\sigma
\dmat\Ga\Gamma
\dmat\Om\Omega
\dmat\ta{\tau}
\dmat\om{\omega}
\dmat\la{\lambda}

\oper\Hom{Hom}
\oper\Ext{Ext}
\oper\GL{GL}
\oper\End{End}
\oper\ch{ch}
\oper\cl{cl}
\oper\rk{rk}
\oper\Aut{Aut}
\oper\Exp{Exp}
\oper\Log{Log}
\oper\Coh{Coh}
\oper\Ind{Ind}
\oper\coker{coker}
\oper\im{im}
\oper\Pic{Pic}
\oper\Res{Res}

\def\oh{\frac12}
\def\iso{\simeq}
\def\ts{\otimes}
\def\wtl#1{\widetilde{#1}}
\def\br{\linebreak}
\def\mto{\mapsto}

\def\inv{^{-1}}
\def\bop{\bigoplus}

\def\nil{{\mathrm{nil}}}
\def\sst{{\mathrm{ss}}}
\def\bun{{\mathrm{vec}}}
\def\gen{{\mathrm{gen}}}

\def\eq#1{\begin{equation}#1\end{equation}}
\def\eql#1#2{\begin{equation}\label{#2}#1\end{equation}}
\def\set#1{\mat{\left\{#1\right\}}}
\def\sets#1#2{\mat{\left\{#1\ \right\vert\left.#2\right\}}}
\def\ang#1{\mat{\left\langle#1\right\rangle}}
\def\rbr#1{\left(#1\right)}
\def\sbr#1{\left[#1\right]}
\def\n#1{\mat{\left\lvert#1\right\rvert}}

\def\sps{\supset}
\def\sbs{\subset}
\def\hlr{\hookleftarrow}
\def\epi{\twoheadrightarrow}

\def\one{\mathbbm{1}}
\def\z{z}
\def\w{w}

\def\Z{\mathbb{Z}}
\def\ur{\bar r}
\def\ud{\bar d}
\def\oE{\bar E} 
\def\oF{\bar F}
\def\oG{\bar G}

\def\ual{\bar\al}

\oper\Nil{\mathbf{Nil}} 
\oper\QS{\mathbf{QS}}
\oper\sCoh{\mathbf{Coh}}
\oper\fCoh{\widetilde{\mathbf{Coh}}}
\oper\stA{\mathbf A} 
\oper\Filt{\mathbf{Flag}} 
\oper\Higgs{\mathbf{Higgs}}
\oper\vol{vol}
\oper\Ker{ker} 
\oper\Im{im}
\oper\Coker{coker}
\def\isom{{\mathrm{iso}}} 
\def\dv{\mid} 
\oper\br{\mathbf r}
\oper\bd{\mathbf d}
\usepackage{hyperref}

\usepackage{todonotes}
\def\note#1{\ \todo[inline,bordercolor=green,backgroundcolor=green!25]{\textbf{Note:} #1}}
\def\note#1{}

\def\vsp#1{} 
\def\vsp{\vspace} 

\begin{document}
\title{Counting Higgs bundles and type $A$ quiver bundles}
\author{Sergey Mozgovoy}
\author[Olivier Schiffmann]{Olivier Schiffmann \dag}
\email{mozgovoy@maths.tcd.ie}
\email{Olivier.Schiffmann@math.u-psud.fr}
\thanks{\dag \; partially supported by ANR grant 13-BS01-0001-01}
\date{\today}
\begin{abstract}
We prove a closed formula counting semistable twisted (or meromorphic) Higgs bundles of fixed rank and degree over a smooth projective curve defined over a finite field of genus $g$, when the degree of twisting line bundle is at least $2g-2$ (this includes the case of usual Higgs bundles). This yields a closed expression for the Donaldson-Thomas invariants of the moduli spaces of twisted Higgs bundles.
We similarly deal with twisted quiver sheaves of type A (finite or affine), obtaining in particular
a Harder-Narasimhan-type formula counting semistable $U(p,q)$-Higgs bundles over a smooth projective curve defined over a finite field. 
\end{abstract}
\maketitle
\tableofcontents

\section{Introduction and statement of results}
\label{sec1}
\vsp{.1in}

\subsection{}
Let $X$ be a smooth projective and geometrically connected curve of genus $g$, defined over a field $\bk$. Let $D$ be a divisor on $X$, of degree $l$. A \textit{$D$-twisted} (or \textit{meromorphic}) Higgs bundle over $X$ is a pair $(E,\te)$, where $E$ is a vector bundle over $X$ and $\te\in\Hom(E,E(D))$. When $D=K$, the canonical divisor of $X$, one recovers the usual notion of Higgs bundles introduced in \cite{Hitchin}. There is a natural notion of semistability for these pairs and one can construct the moduli stack $\Higgs^{\sst}_D(r,d)$ of semistable $D$-twisted Higgs bundles over $X$ of rank $r$ and degree $d$. 

\vsp{.1in}

Despite its importance in algebraic geometry, in the theory of integrable systems and more recently in the theory of automorphic forms, the topology $\Higgs^{\sst}_D(r,d)$ still remains somewhat mysterious. Observe that twisting by a line bundle of degree one yields an isomorphism $\Higgs^{\sst}_D(r,d) \simeq \Higgs^{\sst}_D(r,d+r)$ so that only the value of $d$ in $\Z/r\Z$ matters.
In \cite{HRV} (see also  \cite[Conj.5.6]{hausel_mirrora}), Hausel and Rodriguez-Villegas formulated a precise conjecture for the
Poincar\'e polynomial of $\Higgs^{\sst}_K(r,d)$ when $\bk=\mathbb{C}$ and $\gcd(r,d)=1$. This conjecture was later refined by the first author in \cite[Conj.3]{mozgovoy_solutions} (see also \cite{chuang_motivic}) to a conjecture for the motive $[\Higgs^{\sst}_D(r,d)]$ for any divisor $D$ of degree $l \geq 2g-2$. In the case of $D=K$ this conjecture was verified in low ranks
\cite{garcia-prada_motives,gothen_betti} as well as for the $y$-genus specialization \cite{garcia-prada_y}. Some very interesting results for coprime $r$ and $d$ were also obtained in \cite{chaudouard_sura,chaudouard_sur}. In \cite{schiffmann_indecomposable} the second author gave an explicit formula\footnote{the identification of this formula with the one predicted in \cite{HRV} is still an open problem.} for the Poincar\'e polynomial of $\Higgs^{\sst}_K(r,d)$ when $\bk=\mathbb{C}$ and $\gcd(r,d)=1$, by counting the number of points of $\Higgs^{\sst}_K(r,d)$ over a finite field of high enough characteristic and using the Weil conjectures. This point count in turn relies on a geometric deformation argument to show that (in high enough characteristic) $\n{\Higgs^{\sst}_K(r,d)(\mathbb{F}_q)}=q^{1+(g-1)r^2}\sA_{X,r,d}$, where $\sA_{X,r,d}$ stands for the number of geometrically indecomposable vector bundles on $X$ of rank $r$ and degree $d$. A closed expression for $\sA_{X,r,d}$ is derived in \cite{schiffmann_indecomposable}.

\vsp{.1in}

The main aim of this paper is to generalize the above results to arbitrary meromorphic Higgs bundles (i.e. to $\Higgs^{\sst}_D(r,d)$ for any $D$ of degree $l \geq 2g-2$) and to an arbitrary pair $(r,d)$ (i.e. dropping the coprimality assumption on $r$ and $d$). 
Our approach is in part related to that of \cite{schiffmann_indecomposable}, but it replaces the geometric deformation argument (only available in the symplectic case $K=D$ and in high enough characteristic) by an argument involving the Hall algebra of the category of
meromorphic Higgs bundles, which works for all $D \geq K$, in all characteristics and which yields at the same time the motive of $\Higgs^{\sst}_D(r,d)$ for all $r$ and $d$. We also partly extend these results to the moduli spaces of affine type $A$ quiver bundles (including the moduli spaces of chains of Garcia-Prada et al. on the one hand, and moduli spaces of $U(p,q)$-Higgs bundles on the other).

\vsp{.1in}

Our main result is formulated in terms of the Donaldson-Thomas invariants of the moduli stack $\Higgs_D^\sst(r,d)$. Let $\bk$ be a finite field with $q$ elements, put 
$$\sH_D(r,d)=(-q^\oh)^{-lr^2}\n{\Higgs^{\sst}_D(r,d)(\bk)},$$
and define the DT-invariants $\Om_D(r,d)$ by the formula
\eql{
\sum_{d/r=\ta}\frac{\Om_D(r,d)}{q-1}\w^r\z^d
=\Log\rbr{\sum_{d/r=\ta}\sH_D(r,d)\w^r\z^d},\qquad \ta\in\bR,}{eq:Om00}
where \Log is the plethystic logarithm (see below or e.g.~\cite{mozgovoy_computational}).
It was conjectured in 
\cite{mozgovoy_solutions} that $\Om_D(r,d)$ is a polynomial in the Weil numbers of $X$ which is independent of $d$, regardless of whether $\gcd(r,d)=1$ or not.
Observe that if $\gcd(r,d)=1$ then $\Om_D(r,d)=(q-1)\sH_D(r,d)$, but in general the DT invariant $\Om_D(r,d)$ involves the volume of the stacks
$\Higgs_D^{\sst}(\frac{r}{n},\frac{d}{n})$ for all $n \dv \gcd(r,d)$.
In this paper we give an explicit formula for the invariants $\Om_D(r,d)$ when $\deg D\ge 2g-2$. 

\vsp{.1in}

Before we can state our results, we need to introduce some amount of notation. Let 
$$Z_X(z)=\sum_{d \geq 0} \n{S^dX(\bk)}z^d, \qquad \widetilde{Z}_X(z)=z^{1-g}Z_X(z)$$ 
denote the zeta function of $X$ and its renormalization. 
Given a partition $\lambda=(1^{r_1},2^{r_2},\ldots,t^{r_t})$, we set
$$J_{\lambda}(\z)=\prod_{s \in \lambda} Z^*_X(q^{-1-l(s)}\z^{a(s)})$$
where $a(s)$ and $l(s)$ are respectively the arm and the leg lengths of $s\in\la$ \cite[VI.6.14]{macdonald_symmetric} and
\[
Z^*_X(q^{-1}z)
=\begin{cases}
Z_X(q^{-1}z)\qquad & \text{if } z\ne1,q,\\
\Res_{z=1}Z_X(q^{-1}z)=q^{1-g}\frac{[\Pic^0(X)]}{q-1}\qquad & \text{if }z=1.
\end{cases}
\]

Next, write $n=l(\lambda)=\sum_i r_i$,  
$$r_{<i}=\sum_{k<i} r_k, \qquad r_{>i}=\sum_{k>i}r_k, \qquad r_{[i,j]}=\sum_{k=i}^j r_k$$
and consider the rational function
$$L(\z_n, \ldots, \z_1)=\frac{1}{\prod_{i<j} \widetilde{Z}_X\big(\frac{\z_i}{\z_j}\big)} \sum_{\sigma \in \mathfrak{S}_n} \sigma \left[ \prod_{i<j}
\wtl Z_X\rbr{\frac{\z_i}{\z_j}} \cdot \frac{1}{\prod_{i<n} \rbr{1-q\frac{\z_{i+1}}{\z_i}}} \cdot 
\frac{1}{1-\z_1}\right].$$
Denote by $\Res_{\lambda}$ the operator of taking the iterated residue along
\begin{align*}
&\frac{\z_n}{\z_{n-1}}=\frac{\z_{n-1}}{\z_{n-2}}= \cdots = \frac{\z_{2+r_{<t}}}{\z_{1+r_{<t}}}=q^{-1}\\
&\vdots \qquad \qquad \vdots \qquad  \qquad\qquad \vdots\\
&\frac{\z_{r_1}}{\z_{r_1-1}}= \frac{\z_{r_1-1}}{\z_{r_1-2}}= \cdots = \frac{\z_{2}}{\z_{1}}=q^{-1}.
\end{align*}
Put
$$\widetilde{H}_{\lambda}(\z_{1+ r_{<t}}, \ldots, \z_{1+r_{<i}}, \ldots, \z_1)
=\Res_{\lambda}\left[ L(\z_n, \ldots, \z_1) \prod_{\substack{j =1 \\ j \not\in \{r_{\leq i}\}}}^{n}\frac{d\z_j}{\z_j} \right]$$
and finally
$$H_{\lambda}(\z)=\widetilde{H}_{\lambda}(\z^tq^{-r_{<t}}, \ldots, \z^iq^{-r_{<i}}, \ldots, \z).$$
Note that if $r_i=0$ for some $i$ then the function $\widetilde{H}_{\lambda}$ is independent of its $i$th argument.

\vsp{.1in}

\begin{theorem}[see Proposition~\ref{P:Serre_duality}, Corollary \ref{cor:K0}, Section~6] \label{T:Maain}Let $X$ be a smooth projective geometrically connected curve of genus $g$ defined over a finite field $\bk$, $q=\n{\bk}$.
Let $D$ be a divisor on $X$ and let $l=\deg(D)$. 
Then
\begin{enumerate}
\item For any $r,d$ we have $\Om_K(r,d)=q\sA_{X,r,d}$,
\note{\textbf{Should be removed:} For any $r,d$ and $D$ we have $\Om_D(r,d)=\Om_{K-D}(r,d)$.
This is wrong. I explain it later in Proposition~\ref{P:Serre_duality}.
}
\item Define a series $X_r(z)$ by the formula
$$\sum_{r}X_{r}(z)w^r
=(q-1)\Log\rbr{\sum_\la(-q^\oh)^{l\ang{\la,\la}}
J_\la(\z)H_\la(\z)\w^{\n\la}},$$
where $\ang{\la,\la}=\sum_{i\ge1}\la_i'\la_i'$ and $\la'$ is the partition conjugate to \la.  Then the DT invariants $\Om_D(r,d)$ are given by
$$
\Om_D(r,d)=
\begin{cases}
-\sum_{\xi \in \mu_r}\xi^{-d}\Res_{z=\xi}
X_r(z) \frac{dz}{z},& \deg D>2g-2,\\
-q\sum_{\xi \in \mu_r}\xi^{-d}\Res_{z=\xi}
X_r(z) \frac{dz}{z},& D=K,
\end{cases}
$$
where $\mu_r$ stands for the set of $r$-th roots of unity.
\end{enumerate}
\end{theorem}

\vsp{.1in}

\begin{remark}
Conjecturally, the function $X_r(z)$ has a unique and simple pole at $z=1$, so that $\Om_D(r,d)=[(1-z)X_r(z)]_{z=1}$. It can be shown that $X_r(z)$ is regular outside of $\mu_r$ and has at most simple poles.
\end{remark}

\vsp{.1in}

Statement (i) is an analog of a result of the first author in the context of quivers, see \cite{mozgovoy_motivicb}.
Let us briefly comment on the proof of statement (ii), which is more involved. The standard technique to compute the volume (or DT-invariants) of
the moduli stack of semistable objects in a category, especially when --as in the present case-- there is no freedom of choice for the stability parameter, is to first compute the volume of the moduli stack of \textit{all} objects and then to use some form of Harder-Narasimhan recursion. This strategy can not work in the case of Higgs bundles as the moduli stack $\Higgs_D(r,d)$ is always
of \textit{infinite} volume. Let $\cA_D$ denote the category of $D$-twisted Higgs bundles on $X$. In order to introduce a suitable
truncation $\Higgs^{\geq 0}_D(r,d)$ of $\Higgs_D(r,d)$ we will first define a subcategory $\cA^{\geq 0}_D$ of $\cA_D$ and then consider the moduli stack of semistable objects \textit{in} $\cA_D^{\geq 0}$.
More precisely, let $\Coh^{\geq 0}(X)$ be the category of coherent sheaves over $X$ all of whose HN-factors have slopes $\geq 0$.
Let $\cA^{\geq 0}_D$ be the category of $D$-twisted Higgs sheaves $(E,\te)$ with $E\in \Coh^{\geq 0}(X)$.
It is easy to see that the corresponding moduli stacks $\Higgs^{\geq 0}_D(r,d)$ are of finite volume. We can define the notion of semistability for the objects in this category
and we can construct the moduli stacks $\Higgs^{\geq 0, \sst}_D(r,d)$ of semistable bundles. Note that $\Higgs^{\geq 0, \sst}_D(r,d)$ is \textit{not} a substack of $\Higgs^{\sst}_D(r,d)$, as not all objects in $\Higgs^{\geq 0, \sst}_D(r,d)$ may be semistable in the usual sense.
However, we will show that if $d\ge l\binom r2$ then
$\Higgs^{\sst}_D(r,d)=\Higgs^{\geq 0, \sst}_D(r,d)$. As $\Higgs^{\sst}_D(r,d)\iso\Higgs^{\sst}_D(r,d+r)$, it is thus enough to compute invariants of $\Higgs^{\geq 0,\sst}_D(r,d)$ in order to determine invariants of $\Higgs^{\sst}_D(r,d)$. In particular, we show that if $\deg D\ge 2g-2$
\note{It is important to require here that $\deg D\ge 2g-2$}
and $d\ge l\binom r2$ then
$\Om_D(r,d)=\Om_D^{\geq 0}(r,d)$, where the DT-invariants $\Om_D^{\geq 0}(r,d)$ are defined via (\ref{eq:Om00}) using $\Higgs^{\geq 0,\sst}_D(r,d)$ instead of $\Higgs^{\sst}_D(r,d)$. The volumes of the stacks $\Higgs^{\geq 0,\sst}_D(r,d)$ may be determined by the standard Harder-Narasimhan recursion from the volumes of the stacks $\Higgs^{\geq 0}_D(r,d)$.
Using the formula (see Proposition \ref{P:Serre_duality})
$$\Om^{\ge0}_D(r,d)=\Om^{\ge0}_{K-D}(r,d),$$
we may reduce the case $l > 2g-2$ to the case $l <0$, in which situation all $D$-twisted Higgs bundles are nilpotent.
We then consider a stratification by Jordan types and apply a variant of the method introduced in \cite{schiffmann_indecomposable} to compute the volumes of the stacks $\Higgs^{\geq 0}_D(r,d)$, yielding the formula
$$\sum_{r,d}\Om^{\geq 0}_{D}(r,d)\w^r\z^d
=(q-1)\Log\rbr{\sum_\la(-q^\oh)^{l\ang{\la,\la}}
J_\la(\z)H_\la(\z)\w^{\n\la}}.$$

\vsp{.1in}

The technique developed here is general enough that most of it may be applied to the moduli stacks of type $A$ (twisted) quiver sheaves, and we write the paper in this generality. We note however that, as the Euler form on the category of twisted quiver sheaves
is \textit{not} symmetric unless we are in type $\widehat{A}_{0}$ -- that is, in the Higgs case--, the machinery of Donaldson-Thomas invariants does not apply and we can not obtain as explicit formulas as in the Higgs case.

\vsp{.2in}

\subsection{Plethystic notation} Throughout the paper we will use the standard plethystic operators $\Exp$ and $\Log$, whose definitions we briefly recall here. Consider the space $\mathbb{Q}[[z,w]]$ of power series
in the variables $z,w$. For $l \geq 1$ we define the $l$th Adams operator $\psi_l$ as the $\mathbb{Q}$-algebra map
$$\psi_l:\mathbb{Q}[[z,w]] \to \mathbb{Q}[[z,w]], \qquad  \;z \mapsto z^l, \;w\mapsto w^l.$$
Set $\mathbb{Q}[[z,w]]^+=w\mathbb{Q}[[z,w]] + w \mathbb{Q}[[z,w]]$.
The plethystic exponential and logarithm functions are inverse maps
$$\Exp:\mathbb{Q}[[z,w]]^+ \longrightarrow 1+\mathbb{Q}[[z,w]]^+, \qquad \Log:1+\mathbb{Q}[[z,w]]^+ \longrightarrow \mathbb{Q}[[z,w]]^+$$
respectively defined by
 $$\Exp(f)=\exp \left(\sum_k \frac{1}{k}
\psi_k(f)\right), \qquad \Log(f)=\sum_{k\geq 1} \frac{\mu(k)}{k} \psi_k \left(\log(f)\right).$$
These operators satisfy the usual properties, i.e. $\Exp(f+g)=\Exp(f)\Exp(g)$ and $\Log(fg)=\Log(f)+\Log(g)$.
When taking the plethystic exponential or logarithm of an expression depending on a curve $X$ defined over a finite field $\mathbb{F}_q$
(or on its set of Weil numbers $\{\om_1,\ldots, \om_{2g}\}$) --such as the zeta function $Z_X(z)$ or the Kac polynomials $\sA_{X,r,d}$ for instance--, we understand that the Adams operator $\psi_l$ acts on $X$ by $\psi_l(X)=X \otimes_{\mathbb{F}_q}\mathbb{F}_{q^l}$ (and $\psi_l(\om_i)=\om_i^l$, $\psi_l(q)=q^l$).

\vsp{.2in}

\section{Twisted quiver sheaves}

\vsp{.1in}

\subsection{Definitions}
Let $X$ be a smooth, geometrically connected curve of genus $g$ over a field $\bk$ and let 
$D$ be a divisor on $X$ of degree $l$. 
Given $n\in\bN$, 
let $Q=(I, H)$ be the quiver of type $A_{n-1}^{(1)}$, i.e.\ 
let $I=\Z/n\Z$ be the set of vertices and $H=\sets{i \to i+1}{ i \in\Z/n\Z}$ be the set of arrows.
By definition, a $D$-twisted quiver sheaf (resp.~bundle) on $X$ is a tuple $\oE=( E_i, \theta_i)_{i \in I}$ where $E_i$ is a coherent sheaf (resp.~vector bundle) on $X$ and
$\theta_i \in \Hom(E_i, E_{i+1}(D))$.
As $D$ will be fixed throughout, we will often refer to such a data simply as a quiver sheaf (resp. bundle).

To simplify notation, let $\cA=\Coh(X)^I$ be the category of $I$-graded objects $E=(E_i)_{i\in I}$ in $\Coh(X)$ and consider the shift functor
$$T:\cA\to\cA,\qquad E=(E_i)_i\mto E[1]=(E_{i+1}(D))_i.$$
Then a quiver sheaf can be interpreted as a pair $\oE=(E,\te)$, where $E\in\cA$ and $\te:E\to E[1]$ is a morphism in $\cA$.
We denote by $\cA_D$ the category of quiver sheaves. It is an abelian category, with the obvious notion of morphism. Such categories have been studied by
Garcia-Prada, Gothen and collaborators, see, e.g. \cite{garcia-prada_motives}, \cite{gothen_homological}. Of particular importance are the \textit{Higgs case} ($n=1$) in which
one recovers the category of $D$-twisted (or meromorphic) Higgs sheaves, and the case $n=2$ which, for $\bk=\mathbb{C}$ and $D$ the canonical divisor $K_X$ of $X$ yields a category equivalent to the (collection of) categories of Higgs bundles for the \textit{real groups} $U(p,q)$, see \cite{gothen_real}.
Note also
that as any representation of a finite type $A$ quiver may trivially be regarded as a representation of a cyclic quiver, the categories of quiver sheaves considered here also contain the categories of quiver sheaves for finite type $A$ quivers (also known as 'chains', see \cite{garcia-prada_motives}).

\vsp{.1in}

For $L$ a line bundle on $X$ and $E\in\cA$, define $E\ts L=(E_i\ts L)_{i}$.
Similarly, for $\oE=(E,\te)\in\cA_D$, we define $\oE\ts L=(E\ts L,\te\ts L)$ and we use a similar notation for the operation of shifting by a divisor.
Similarly, we define $\oE[1]=(E[1],\te[1])$.

\vsp{.1in}

For a coherent sheaf $E \in \Coh(X)$, we define its \textit{class} to be the pair $\cl E=(\rk E, \deg E) \in \mathbb{Z}^2$.
The \textit{slope} of a sheaf is 
$$\mu(E)=\frac{\deg E}{\rk E} \in \mathbb{Q} \cup \{\infty\}.$$
Similarly, for $E=(E_i)_i\in\cA$, we define
\begin{gather}
\cl E=(\cl E_i)_i\in\Ga=(\bZ^2)^I=\bZ^I\oplus\bZ^I,\label{E:classdef}\\
\mu(E)=\frac{\deg E}{\rk E}\in\bQ\cup\set\infty,\qquad
\rk E=\sum\rk E_i,\qquad \deg E=\sum\deg E_i.
\label{E:slopequiversheaves}
\end{gather}
For any $\ga=(\ur,\ud)\in\Ga=\bZ^I\oplus\bZ^I$,
define 
$$\ga[1]=\rbr{r_{i+1},d_{i+1}+lr_{i+1}}_i\in\Ga.$$
Then $\cl(E)[1]=\cl(E[1])$ for any $E\in\cA$.

We extend this notation to quiver sheaves by setting
$\cl \oE=\cl E$ and $\mu(\oE)=\mu(E)$, for $\oE=(E,\te)\in\cA_D$.
We will write $\cA_D(\ur,\ud)$ for the subcategory of quiver sheaves $\oE$ of class $(\ur,\ud) \in (\Z^2)^I$.

\vsp{.1in}

For $E,F\in\Coh(X)$, we denote by $\chi(E,F)$ the Euler form
on the category $\Coh(X)$,
i.e.~we set
$$\chi(E,F)=\dim \Hom(E,F)-\dim \Ext^1(E,F).$$
By the Riemann-Roch formula,
$$\chi(E,F)=(1-g)\rk E\cdot \rk F + (\rk E\cdot \deg F-\rk F\cdot \deg E).$$
Since $\chi(E,F)$ only depends on $\cl E$ and $\cl F$ we will sometime denote this Euler form also by $\chi(\cl E, \cl F)$.
The same notation is used for the Euler form on the category $\cA$.

\vsp{.2in}

\subsection{Homological properties}
The categories $\Coh(X)$ and $\cA$ are of cohomological dimension $1$, while the category $\cA_D$ is of homological dimension $2$.
More precisely, we have

\begin{theorem}[cf.~Gothen-King {\cite{gothen_homological}}]
Given $\oE=(E,\te)$, $\oF=(F,\te')$ in $\cA_D$,
there is a long exact sequence
\begin{multline*}
0\to\Hom(\oE,\oF)\to\Hom(E,F)\to\Hom(E,F[1])\\
\to\Ext^1(\oE,\oF)\to\Ext^1(E,F)\to\Ext^1(E,F[1])
\to\Ext^2(\oE,\oF)\to0
\end{multline*}
and the groups $\Ext^i(\oE,\oF)$ vanish for $i>2$.
\end{theorem}

\vsp{.1in}

Let us denote by 
$$\chi_D(\oE, \oF)=\dim \Hom(\oE,\oF)-\dim \Ext^1(\oE,\oF) + \dim \Ext^2(\oE,\oF)$$
the Euler form in $\cA_D$.

\vsp{.1in}

\begin{corollary}\label{cr:chi}
For any $\oE,\oF\in\cA_D$, we have
\begin{enumerate}
	\item $\hi_D(\oE,\oF)
	=\hi(E,F)-\hi(E,F[1])
	=\sum_i \left( \hi(E_i,F_i)-\hi(E_i,F_{i+1}(D))\right)$.
	\item if $n=1$ then $\hi_D(\oE,\oF)=\hi_D(\oF,\oE)=-l\rk E\cdot \rk F$.
\end{enumerate}
\end{corollary}

Observe that the Euler form $\chi$ on $\cA_D$ is symmetric only in the case of Higgs sheaves (i.e. for $n=1$).
Applying Serre duality for coherent sheaves, we obtain the following form of Serre 
duality for quiver sheaves.

\begin{corollary}\label{C:SD}
For any $\oE,\oF\in\cA_D$, we have
\[\Ext^i(\oE,\oF)\iso\Ext^{2-i}(\oF,\oE[-1](K_X))^*.\]
\end{corollary}

\note{Using notation $E[1]=(E_{i+1}(D))_i$ defined later, one can write $(\om\cdot \oE)(K_X-D)=\oE[-1](K_X)$.}

\vsp{.1in}

Recall that a coherent sheaf $E \in \Coh(X)$ is called semistable (resp. stable) if for any proper subsheaf $F \subset E$ we have $\mu(F) \leq \mu(E)$
(resp. $\mu(F) < \mu(E)$). The Harder-Narasimhan (HN for short) filtration of $E$ is the unique filtration
$$E=E_1\sps E_2\sps\dots\sps E_s\sps E_{s+1}=0$$
such that $E_i/E_{i+1}$ are semistable and 
$$\mu(E_1/E_2) < \mu(E_2/E_3) < \dots < \mu(E_s).$$
We set
\begin{gather*}
\si(E)=\set{\mu(E_1/E_2), \mu(E_2/E_3), \ldots, \mu(E_s)}\sbs\bQ\cup\set\infty,\\
\mu_{\min}(E)=\min\si(E)=\mu(E_1/E_2),\qquad
\mu_{\max}(E)=\max\si(E)=\mu(E_s).
\end{gather*}
In the same way we define semistable objects in \cA using the slope function (\ref{E:slopequiversheaves}).
An object $E=(E_i)_i\in\cA$ is semistable if and only if all $E_i$ are semistable and have equal slope.
We similarly define semistable objects in $\cA_D$ using the slope function (\ref{E:slopequiversheaves}),
i.e.\ we say that $\oE$ is semistable if for any quiver subsheaf $\oF \subset \oE$ we have $\mu(\oF) \leq \mu(\oE)$. We further say that $\oE$ is stable if the inequality is strict for any proper quiver subsheaf $\oF$. For $\nu \in \mathbb{Q} \cup \{\infty\}$ let us denote by $\cA_D^{(\nu)}$ the full subcategory of $\cA_D$ whose objects are the semistable quiver sheaves of slope $\nu$, and $\cA^{\sst}_D(\ur,\ud)$ for the full subcategory of $\cA_D(\ur,\ud)$ whose objects are semistable.

\begin{remark}
If $l \leq 0$ then a quiver sheaf $\oE=(E,\te)$ is semistable if and only if $E\in\cA$ is semistable; indeed, if $E$ is not semistable then the last term $E_s$ in its HN-filtration
satisfies $\te(E_{s})\sbs E_s[1]$ for slope reasons, and thus $(E_s,\theta|_{E_s})$
is automatically a (destabilizing) quiver subsheaf of $\oE$. This shows that the notion of semistable quiver sheaf is only interesting when $l >0$.
\end{remark}

We summarize the standard properties of $\cA_D$ with respect to the above semistability notion in the following Proposition, whose proof is left to the reader:

\vsp{.1in}

\begin{proposition} The following hold:
\begin{enumerate}
\item For any $\nu$, $\cA_D^{(\nu)}$ is an abelian subcategory of $\cA_D$ which is stable under extensions and direct summands.
\item For any line bundle $L$ on $X$, twisting by $L$ defines an equivalence $\cA_D^{(\nu)} \simeq
\cA_D^{(\nu + \deg(L))}$, 
\item If $\nu > \gamma$ then $\Hom(\cA_D^{(\nu)}, \cA_D^{(\gamma)})=0$,
\item Any quiver sheaf $\oE$ carries a unique filtration 
$$\oE=\oE_1\sps \oE_2\sps \oE_3\sps\cdots\sps E_s\sps E_{s+1}=0$$
whose factors $\oE_i / \oE_{i+1}$ are semistable and such that
$$\mu(\oE_1 / \oE_{2}) < \mu(\oE_2 / \oE_3) < \cdots.$$ 
\end{enumerate}
\end{proposition}

\vsp{.1in}

The following result will be crucial for our purposes.

\vsp{.1in}

\begin{corollary}\label{cr:sst zero}
Assume that $l \ge 2g-2$ and $\oE,\oF\in\cA_D$ are semistable objects such that $\mu(\oE)<\mu(\oF)$.
Then $\Ext^2(\oE,\oF)=0$.
\end{corollary}
\begin{proof}
By our assumption $\deg(K-D)\le0$.
Therefore
\[\mu(\oE[-1](K))\le\mu(\oE)<\mu(\oF).\]
By the semistability of $\oE[-1](K)$ and $\oF$ we conclude that
\[\Hom(\oF,\oE[-1](K))=0.\]
By the Serre duality of Corollary~\ref{C:SD}, this implies that $\Ext^2(\oE,\oF)=0$.
\end{proof}

\vsp{.2in}

\subsection{Positive quiver sheaves}
In this paragraph we introduce a suitable truncation $\cA_D^{\geq 0}$ of $\cA_D$ and prove that, for large slopes, 
$\cA_D$ and $\cA_D^{\geq 0}$ have the same semistable objects. 

\vsp{.1in}

We denote by $\Coh^{\geq 0}(X)$ the full subcategory of $\Coh(X)$ whose objects verify $\mu_{\min}(E) \geq 0$. The subcategory
$\Coh^{\geq 0}(X)$ is closed under extensions and quotients, but not under taking subobjects.
Similarly, we define $\cA^{\ge0}\sbs\cA$ to be the subcategory of \cA whose objects $E\in\cA$ satisfy $\mu_{\min}(E)\ge0$.
We define $\cA_D^{\geq 0}$ as  the full subcategory of $\cA_D$ whose objects $\oE=(E,\te)$ verify $E\in\cA^{\ge0}$.
Obviously, if $\oE \in \cA_D^{\geq 0}$ then $\mu(\oE) \geq 0$. 

We will say that an object $\oE \in \cA_D^{\geq 0}$ is  \textit{semistable in} $\cA_D^{\geq 0}$ if, for any $\oF \subset \oE$ in $\cA_D^{\geq 0}$,
we have $\mu(\oF) \leq \mu(\oE)$.
We define the notion of a stable object of $\cA_D^{\geq 0}$ accordingly, replacing $\leq$ by $<$. Observe that a semistable object in $\cA_D^{\geq 0}$ may be unstable
in the usual sense, but the converse is false: an object of $\cA_D^{\geq 0}$ which is semistable in the usual sense is also semistable
in $\cA_D^{\geq 0}$.
We denote by $\cA_D^{\geq 0, (\nu)}$ the subcategory of quiver sheaves in $\cA_D^{\geq 0}$ of slope $\nu$ which are semistable in $\cA_D^{\geq 0}$.
Therefore we have an inclusion of subcategories
$$\cA_D^{\geq 0} \cap \cA_D^{(\nu)} \subseteq \cA_D^{\geq 0, (\nu)}$$
which is strict in general. Note, however, that it is an equality if $l\leq 0$. The full subcategory $\cA^{\geq 0, \sst}_{D}(\ur,\ud)$ of $\cA^{\geq 0}_{D}(\ur,\ud)$ is defined in the same way.

\vsp{.1in}

\begin{proposition}\label{P:1}
For any $r\ge1$ and $\nu\ge l\frac{r-1}2$, we have
$$\rbr{\oE \in \cA_D^{\geq 0,(\nu)},\,\rk(\oE)=r}
\implies
\rbr{\oE\in\cA_D^{(\nu)}}.$$
\end{proposition}
\begin{proof} We may assume that $l >0$. We begin with the following observation.

\begin{lemma}\label{L:1}
Let $E\in\cA$ and assume that
$\si(E)=\set{\nu_1<\dots<\nu_s}$ has a gap $\nu_{k+1}-\nu_{k}>l$.
Then there exists no $\te\in\Hom(E,E[1])$
such that $\oE=(E,\te)\in\cA_D$ is semistable.
\end{lemma}
\begin{proof}
Let $\te\in\Hom(E,E[1])$ and $\oE=(E,\te)$.
There exists a (unique) short exact sequence
$$0\to E'\to E\to E''\to 0$$
in \cA with $\mu_{\min}(E')\geq \nu_{k+1}$ and $ \mu_{\max}(E'')\le \nu_k$.
Since
$\mu_{\min}(E')>\mu_{\max}(E'')+l$
we deduce that $\Hom(E',E''[1])=0$ and thus
$\te(E')\sbs E'[1]$.
But then $\oE'=(E',\te)$ is a destabilizing subobject of~$\oE$.
\end{proof}

\begin{lemma}\label{L:2}
Let $\oE=(E,\te)\in\cA_D$ be semistable.
Assume that $d \geq l \binom{r}{2}$, where
$d=\deg(E)$ and $r=\rk(E)$.
Then $E\in\cA^{\ge0}$.
\end{lemma} 
\begin{proof}
Let us write 
$\si(E)=\set{\nu_1<\nu_2<\ldots<\nu_s}$
and let $r_k,d_k$ be the rank and degree of the $k$-th factor of the HN filtration of $E$, so that
$$\sum_k d_k =d, \qquad \sum_k r_k=r, \qquad \nu_k=\frac{d_k}{r_k}.$$
By Lemma~\ref{L:1} we have
$\nu_{k+1}-\nu_{k} \leq l$ hence $r_k\ge1$ for all $k$,  $\nu_k \leq \nu_1 + (k-1) l$ and
$$l\frac{(r-1)}2\leq\frac{d}{r}
=\frac{\sum \nu_k r_k}{r}
\le\frac{\sum_{k}(\nu_1+(k-1)l)r_k}{r}
\le \nu_1+l\frac{(r-1)}{2}
$$
which implies that $\nu_1 \geq 0$, hence $E\in\cA^{\ge 0}$.
We used here the fact that if $\sum_k r_k =r$
and $r_k\ge1$ for all $k$ then $\sum_k (k-1)r_k \leq \binom{r}{2}$.  
\end{proof}

\vsp{.1in}

We may now finish the proof of Proposition~\ref{P:1}.
Let $\oE=(E,\te)\in\cA_D^{\ge0}$ be an object of rank $r$ and slope $\mu(E) \ge l \frac{r-1}{2}$, semistable in $\cA_D^{\ge0}$.
Assume that $\oE$ is not semistable in the usual sense.
Then $\oE$ has a destabilizing subobject $\oF=(F,\te)\in\cA_D$ of rank $r'\leq r$.
Therefore
$$\mu(\oF) > \mu(\oE) \geq l\frac{r-1}{2} \geq l \frac{r'-1}{2}.$$
By Lemma~\ref{L:2}, $F\in\cA^{\ge0}$.
But then $\oF$ is a destabilizing subobject of $\oE$
\textit{in} $\cA_D^{\geq 0}$, contradicting the assumption on $\oE$. Proposition~\ref{P:1} is proved.
\end{proof}

\vsp{.2in}

\subsection{Notations for stacks}
\label{sec:notation stacks}
Let us denote by $\sCoh(r,d)$ the stack of coherent sheaves of rank $r$ and degree $d$ on $X$. This stack is locally of finite type and of finite volume. Let $\sCoh^{\geq 0}(r,d)$ be the substack parametrizing positive coherent sheaves. This open substack is of finite type.
Similarly, given $(\ur,\ud)\in(\bZ^2)^I=\bZ^I\oplus\bZ^I$,
let $\stA(\ur,\ud)$ be the stack of objects in $\cA$ having class 
$(\ur,\ud)$.
Let $\stA^{\ge0}(\ur,\ud)$ be the substack parametrizing positive objects.
We denote by $\QS_D(\ur,\ud)$ the stack parametrizing quiver sheaves on $X$ of class $(\ur,\ud)$.
It is again a stack locally of finite type, but it is of infinite volume in general.
The open substack parametrizing positive quiver sheaves is denoted $\QS_D^{\ge0}(\ur,\ud)$.
Contrary to $\QS_D(\ur,\ud)$, this stack is of finite type and of finite volume.
Let $\QS_D^{\sst}(\ur,\ud)\sbs \QS_D(\ur,\ud)$
(respectively, $\QS_D^{\ge0,\sst}(\ur,\ud)\sbs \QS_D^{\ge0}(\ur,\ud)$) be the substack of quiver sheaves semistable in $\cA_D$ (respectivly, in $\cA_D^{\ge0}$).
In the Higgs case, we denote these stacks by $\Higgs_D(r,d)$,  $\Higgs_D^{\geq 0}(r,d)$,
$\Higgs_D^\sst(r,d)$,  and $\Higgs_D^{\geq 0,\sst}(r,d)$ respectively.

\vsp{.2in}

\section{Generating functions and Donaldson-Thomas invariants}
\label{sec3}
\vsp{.1in}

In this section we introduce several generating functions for the volume of the stacks of positive and/or semistable quiver sheaves, as well as the Donaldson-Thomas invariants of the categories $\cA_D,\, \cA_D^{\geq 0}$ in the Higgs case. We begin with a brief
review of the relevant theory of Hall algebras. Let us from now on assume that the curve $X$ is defined over a finite field $\bk=\mathbb{F}_q$, and set $X_k=X \otimes_{\mathbb{F}_q} \mathbb{F}_{q^k}$ for any $k\ge1$.

\vsp{.1in}

\subsection{Hall algebras and quantum torus.}
Let \cA be an abelian category, linear over a finite field $\bk=\mathbb{F}_q$, of finite homological dimension and such that $\dim \Ext^k(M,N)<\infty$ for all objects $M,N$ and all $k \geq 0$. Let $\chi : K_0(\cA) \otimes_{\Z} K_0(\cA) \to \Z$ denote the Euler form.
Let also $\Ga$ be a lattice equipped with a skew-symmetric form \ang{-,-}
and with a group homomorphism $\cl:K_0(\cA)\to\Ga$ such that
\[\hi(E,F)-\hi(F,E)=\ang{\cl E,\cl F},\qquad E,F\in\cA.\]
The algebra $\bT=\bQ(q^\oh)[\Ga]$ equipped with the product
\[e^\al\circ e^\be=(-q^\oh)^{\ang{\al,\be}}e^{\al+\be},\qquad \al,\be\in\Ga\]
is called the \textit{quantum (affine) torus}. Let $\cH$ be the Hall algebra of $\cA$ (see e.g. \cite{Slectures}).
Both $\cH$ and $\bT$ are graded by the lattice $\Ga$. We will occasionally consider their completions
$$\prod_{\alpha \in \Ga} \cH[\alpha], \qquad \prod_{\alpha \in \Ga} \mathbb{T}[\alpha]$$
which we still denote by $\cH$ and $\bT$ respectively for simplicity when there is no risk of confusion.
One defines the integration map
\eql{I:\cH\to\bT,\qquad [E]\mto (-q^\oh)^{\hi(E,E)}\frac{e^{\cl E}}{\n{\Aut E}}.}{int map}
A crucial property of $I$ is that it is a ring homomorphism if \cA has homological dimension one \cite{reineke_counting}.
More generally, it satisfies
\[I([E]\circ[F])=I(E)\circ I(F)\]
if $\Ext^k(F,E)=0$ for $k\ge2$. This explains the significance of Cor.\ \ref{cr:sst zero}.

\vsp{.1in}

\subsection{Generating functions}
\label{sec:invar}
We will denote the Hall algebra of $\cA_D$ by $\cH_D$. 
Set $\Ga=(\mathbb{Z}^2)^I=\bZ^I\oplus\bZ^I$, consider the map
$\cl: K_0(\cA_D) \to \Ga$ defined in (\ref{E:classdef}),
and equip $\Ga$ with bilinear forms
\begin{gather*}
\chi_D(\ga,\ga')
=\sum_{i\in I}\rbr{
(1-g) r_ir'_i + (g-1-l)r_ir'_{i+1}
+\begin{vmatrix} r_i & r'_i-r'_{i+1} \\ d_i & d'_i-d'_{i+1}\end{vmatrix}},\\
\ang{\ga,\ga'}
=\hi_D(\ga,\ga')-\hi_D(\ga',\ga)
=\sum_{i \in I}
\rbr{
(g-1-l)r_i(r'_{i+1}-r'_{i-1})
+\begin{vmatrix} r_i & 2r'_i-r'_{i-1}-r'_{i+1}\\ d_i & 2d'_i -d'_{i-1}-d'_{i+1}\end{vmatrix}}
\end{gather*}

for $\ga=(\ur,\ud)\in\Ga=\bZ^I\oplus\bZ^I$ and $\ga'=(\ur',\ud')\in\Ga=\bZ^I\oplus\bZ^I$.
Observe that when $n=1$ (i.e.\ in the Higgs case) the form $\ang{-,-}$ vanishes hence the quantum torus
$\bT=\bQ(q^\oh)\sbr{\bZ^2}$ is commutative.

We will use variables
\[\z^{\ud}=e^{(0,\ud)},\qquad \w^{\ur}=e^{(\ur,0)}, \qquad (\ur,\ud) \in \Ga.\]

Let $\ga=(\ur,\ud)\in\Ga$.
Recall that $\cA_D^\sst(\ga)=\cA_D^\sst(\ur,\ud)$ is the (finite) set of isomorphism classes
of semistable objects $\oE=(E,\te)\in\cA_D$ with $E$ having class $\ga$.
Note that if $(E,\te)$ is semistable and has positive rank then $E=(E_i)_i$ is an $I$-graded vector bundle.
Define
\begin{gather}
\one^\sst_{\ga}
=\sum_{\oE\in\cA^\sst_D(\ga)}[\oE]\in\cH_D,\\
\sH_D(\ga)e^\ga
=I\rbr{\one^\sst_{\ga}}
=\sum_{\oE\in\cA^\sst_D(\ga)}(-q^\oh)^{\hi_D(\ga,\ga)}\frac{1}{\n{\Aut\oE}}e^\ga.
\end{gather}
Tensoring by a line bundle preserves semistability; from this it is easy to see that $\sH_D(\ur,\ud)=\sH_D(\ur,\ud+\ur)$.

\vsp{.1in}

We likewise define the elements
\begin{gather}
\one^{\geq 0,\sst}_{\ga}=\sum_{\oE\in\cA^{\geq 0,\sst}_D(\ga)}[\oE]\in\cH_D,\\
\sH^{\geq 0}_D(\ga)e^\ga
=I\rbr{\one^{\geq 0,\sst}_{\ga}}
=\sum_{\oE\in\cA^{\geq 0,\sst}_D(\ga)}(-q^\oh)^{\hi_D(\ga,\ga)}\frac{1}{\n{\Aut\oE}}
e^\ga
\end{gather}
and
\begin{gather}\label{E:defI}
\one^{\geq 0}_{\ga}
=\sum_{\oE\in\cA^{\geq 0}_D(\ga)}[\oE]\in\cH_D,\\
\sI^{\geq 0}_D(\ga)e^\ga
=I\rbr{\one^{\geq 0}_{\ga}}
=\sum_{\oE\in\cA^{\geq 0}_D(\ga)}(-q^\oh)^{\hi_D(\ga,\ga)}\frac{1}{\n{\Aut\oE}}
e^\ga.
\end{gather}
Observe that the categories $\cA^{\geq 0}_D(\ga)$, and hence \textit{\`a fortiori} the categories $\cA^{\geq 0,\sst}_D(\ga)$, have finitely many objects up to isomorphisms so that the above sums are well-defined.

\vsp{.1in}

The uniqueness of the Harder-Narasimhan filtration implies the following identity in the Hall algebra:
\eq{\sum_{\ga}\one^{\geq 0}_{\ga}=\prod_{\ta\downarrow}
\rbr{\sum_{\mu(\ga)=\ta}\one^{\geq 0,\sst}_{\ga}},}
where the product is taken in the decreasing order of $\ta\in[0,+\infty]$.
If $l=\deg D\ge 2g-2$ then, by Corollary \ref{cr:sst zero}, the integration map $I:\cH_D\to\bT$ preserves the product
on the right.
Therefore we obtain
\eq{\label{HN2}
\sum_{\ur,\ud}\sI^{\geq 0}_D(\ur,\ud)\w^{\ur} \z^{\ud}=\prod_{\ta \downarrow}
\rbr{\sum_{\mu(\ur,\ud)=\ta}\sH^{\geq 0}_D(\ur,\ud)\w^{\ur} \z^{\ud}}.}
Note that unless $n=1$, the product on the right is ordered as the quantum torus is not commutative. In sections~4 and 5 we will see
how to compute the volumes $\sI^{\geq 0}_D(\ur,\ud)$ for $l \geq 2g-2$. This will allow us to determine in Section~\ref{sec:HN recursion}, via
a Harder-Narasimhan recursion, the volumes $\sH^{\geq 0}_D(\ur,\ud)$ of the stacks of semistable positive quiver sheaves and thus, by passing to a limit as $d \to \infty$, to determine the volumes $\sH_D(\ur,\ud)$ of the stacks of semistable  quivers sheaves.

\vsp{.1in}

\subsection{DT invariants.} The special case $n=1$ is the most important as it corresponds to the moduli stacks of (meromorphic) Higgs bundles. In that situation, $\mathbb{T}$ is commutative, $\Ga = \Z^2$, and we may define the Donaldson-Thomas invariants $\Om_D(r,d)$ by the following formula
\eq{\label{eq:Om1}
\sum_{d/r=\ta}\frac{\Om_D(r,d)}{q-1}\w^r\z^d
=\Log\rbr{\sum_{d/r=\ta}\sH_D(r,d)\w^r\z^d},
\qquad\ta\in\bR}
where $\Log$ is the plethystic logarithm.
\note{As I commented in other place, we can define $\Om_D$ for any $D$, but I think that it is meaningful only for $\deg D\ge 2g-2$.}
Comparing equations for $\ta$ and $\ta+1$, we obtain that $\Om_D(r,d)=\Om_D(r,d+r)$.
Various tests justify the conjecture that if $\deg D\ge 2g-2$, then $\Om_D(r,d)$ are independent of $d$ (cf.\ \cite[Conj.1.9]{chuang_motivic}).
Note that if $r,d$ are coprime then
\[\sH_D(r,d)=\frac{\Om_D(r,d)}{q-1}.\]
For $D=K_X$ and coprime $r,d$,
independence of $\sH_D(r,d)$ of $d$ was conjectured in \cite[Conj.3.2]{hausel_mirrora}.

\vsp{.1in}

We also consider the truncated version
\eq{\label{eq:correct DT}
\sum_{r,d}\frac{\Om^{\geq 0}_D(r,d)}{q-1}\w^r\z^d
=\Log\rbr{\sum_{r,d}\sI^{\geq 0}_D(r,d)\w^r\z^d}.}
If $\deg D\ge 2g-2$, then we obtain from \eqref{HN2} that
\eq{\label{eq:Om2}
\sum_{d/r=\ta}\frac{\Om^{\geq 0}_D(r,d)}{q-1}\w^r\z^d
=\Log\rbr{\sum_{d/r=\ta}\sH^{\geq 0}_D(r,d)\w^r\z^d},
\qquad\ta\in\bR.}

\begin{lemma}
\label{lm:Om equality}
If $\deg D\ge 2g-2$ and $d\ge l\binom r2$, then
$\Om^{\geq 0}_D({r},{d})=\Om_D({r},{d})$.
\end{lemma}
\note{Again, one can define $\Om_D^{\ge0}$ using $\sH_D^{\ge0}$ and get $\Om^{\geq 0}_D({r},{d})=\Om_D({r},{d})$ for arbitrary $D$.
But then we can not express $\Om_D^{\ge0}$ in terms of $\sI_D^{\ge0}$ if $\deg D<2g-2$.
This would imply that we can not relate $\Om^{\ge0}_D$ and $\Om^{\ge}_{K-D}$.}

\begin{proof}
By Proposition~\ref{P:1}, we have $\sH^{\geq 0}_D({r},{d})=\sH_D({r},{d})$ for $d\ge l\binom r2$.
Applying formulas \eqref{eq:Om1} and \eqref{eq:Om2}, we obtain
$\Om^{\geq 0}_D({r},{d})=\Om_D({r},{d})$ for $d\ge l\binom r2$.
\end{proof}

In Section~\ref{sec6} we will use
our computation of $\sI_D^{\geq 0}(r,d)$ for negative $D$ (see Section~\ref{sec: nilp}) to give a closed expression for the truncated DT-invariants $\Om^{\geq 0}_D(r,d)$. Because $\Om_D(r,d)=\Om_D(r,d+r)$ this will be enough to fully determine the DT-invariants
$\Om_D(r,d)$.

\vsp{.2in}

\section{Serre duality and nilpotent quiver sheaves.}
\label{sec4}
\vsp{.1in}

In this section, we will show by some simple Serre duality argument that the computation of the volume of the stacks $\QS^{\geq 0}_{D}(\alpha)$ is equivalent to the computation of the volume of stacks $\QS_{K-D}^{\geq 0}(\alpha)$ where $K$ is the canonical divisor of $X$. This will allow us to relate, \textit{when} $l \geq 2g-2$, the volume of 
$\QS_{D}^{\geq 0}(\alpha)$ to the volume of certain stacks parametrizing \textit{nilpotent} quivers sheaves.

\vsp{.1in}

\subsection{Consequences of Serre duality.}

\begin{proposition}\label{P:Serre_duality}
For any $D$ and any $\ga\in \Ga$,
we have $\sI^{\geq 0}_D(\ga)=\sI^{\geq 0}_{K-D}(\ga^*)$,
where $\ga^*=(\ga_{-i})_i\in\Ga$.
In the Higgs case, we have
$\Om^{\geq 0}_{D}(\ga)=\Om^{\geq 0}_{K-D}(\ga)$. \end{proposition}
\note{$\Om^{\geq 0}_{D}(r,d)=\Om^{\geq 0}_{K-D}(r,d)$ is true only if we define them using $\sI^{\ge0}$ (and not $\sH^{\ge0}$).
But then we can not relate $\Om_D^{\ge0}$ to $\Om_D$ if $\deg D<2g-2$.
For this reason we can not say that $\Om_{D}(r,d)=\Om_{K-D}(r,d)$ ($\Om_D$ for $\deg D<2g-2$ does not make much sense anyway, unless we can find a meaningful definition??).
}
\begin{proof}

We have, by definition,
$$\sI^{\ge0}_D(\ga)
=(-q^{\oh})^{\hi_D(\ga,\ga)}\sum_{\oE \in \cA^{\ge0}_D(\ga)}
\frac{1}{|\Aut(\oE)|}
=(-q^{\oh})^{\hi_D(\ga,\ga)}\sum_{E\in\cA^{\ge0}(\ga)}\frac{q^{h^0(E, E[1])}}{|\Aut(E)|},$$
where we have set
$$h^k(E,F)=\dim\Ext^k(E,F), \qquad k=0,1;\ E,F \in \cA.$$
Given $E=(E_i)_i\in\cA^{\ge0}(\ga)$, consider $F=(E_{-i})_i\in\cA^{\ge0}(\ga^*)$.
Then $(F,\vi)\in\cA^{\ge0}_{K-D}(\ga^*)$ means that
$$\vi\in\prod_i \Hom(E_i,E_{i-1}(K-D))=\Hom(E,E[-1](K)).$$
Therefore we have to prove that
$$\hi_D(\ga,\ga)+2h^0(E,E[1])=\hi_{K-D}(\ga^*,\ga^*)+2h^0(E,E[-1](K))$$
or equivalently, by Serre duality, that
\eq{\label{eq:dual1}
\hi_D(\ga,\ga)+2\hi(E,E[1])=\hi_{K-D}(\ga^*,\ga^*).}
By Corollary~\ref{cr:chi}, we have
\begin{gather*}
\hi_D(\ga,\ga)=\hi(E,E)-\hi(E,E[1]),\\
\hi_{K-D}(\ga^*,\ga^*)=\hi(E,E)-\hi(E,E[-1](K)).
\end{gather*}
This and the fact that $\chi(E,F)=-\chi(F,E(K))$, for any $E,F\in\cA$, imply \eqref{eq:dual1}.
The statement concerning Higgs bundles follows
from the definition of the DT-invariants
\eqref{eq:correct DT}.
\end{proof}

\vsp{.1in}

\subsection{Nilpotent quiver sheaves.}
We will say that a quiver sheaf $\oE=(E,\te)$ is \textit{nilpotent} if
there exists $s>0$ such that the composition $\te^s$
$$E\to E[1]\to E[2]\to\to\dots\to E[s]$$
vanishes.
We call the minimal $s$ satisfying this property the \textit{nilpotency index} of $\oE$.

\vsp{.1in}

Let us denote by $\QS^{\nil}_D(\ga)$ the stack of $D$-twisted nilpotent quiver sheaves of class $\ga\in\Ga$.
We also denote by $\QS_D^{\ge0,\nil}(\ga)$
the open substack parametrizing quiver sheaves $\oE$  belonging to $\cA_D^{\geq 0}$.
Observe that if $l<0$ then any quiver sheaf is automatically nilpotent, i.e. 
\begin{equation}\label{E:QSNil}
\QS_D(\ga)=\QS^{\nil}_D(\ga), \qquad \QS^{\geq 0}_D(\ga)=\QS^{\geq 0,\nil}_D(\ga)\qquad (\forall\;l <0).
\end{equation}

For any $D$ we may define $\sI^{\geq 0, \nil}_{D}(\ga)$ just like in (\ref{E:defI}),
and in the Higgs case we may also define
$\Om^{\nil}_D(r,d)$ like in \eqref{eq:Om1}
and $\Om^{\geq 0,\nil}_D(r,d)$ like in \eqref{eq:correct DT}.
From (\ref{E:QSNil}) and \eqref{eq:dual1} we immediately deduce the following

\vsp{.1in}

\begin{corollary}\label{C:SDQS}
If $l > 2g-2$ then, for any $\ga=(\ur,\ud) \in \Ga$,
we have $\sI^{\geq 0}_D(\ga)=\sI^{\geq 0,\nil}_{K-D}(\ga^*)$, or equivalently
$$\vol(\QS^{\geq 0}_D(\ga)(\bk))=q^{\hi(\ga,\ga[1])}\vol(\QS_{K-D}^{\geq 0,\nil}(\ga^*)(\bk)),$$
where
$$\hi(\ga,\ga[1])
=(1-g+l)\sum_i r_i r_{i+1} + \begin{vmatrix} r_i & r_{i+1} \\ d_i & d_{i+1} \end{vmatrix}.$$
In the Higgs case we have 
$\Om_D^{\geq 0}(r,d)=\Om_{K-D}^{\geq 0, \nil}(r,d)$.
\end{corollary}

\vsp{.1in}

\subsection{From Higgs sheaves to nilpotent Higgs sheaves.} The aim of this section is to prove a result somewhat similar to Corollary~\ref{C:SDQS} in the critical case $D=K$.

\vsp{.1in}

 We begin with the Higgs case, for which things can be made very explicit
in terms of Donaldson-Thomas invariants and Kac polynomials of curves. Let $\sA_{X,r,d}$ denote the number of absolutely indecomposable coherent sheaves on $X$ of rank $r$ and degree $d$. Similarly, let $\sA^{\geq 0}_{X,r,d}$ denote the number of positive (that is, contained in $\cA^{\ge0}$) absolutely indecomposable vector bundles of rank $r$ and degree $d$. Both of these numbers are the evaluation, at the collection of Weil numbers of $X$, of certain polynomials determined in  \cite{schiffmann_indecomposable} which only depend on the genus of $X$. For simplicity, we will drop the index $X$ from the notation when the curve is understood.

\vsp{.1in}

\begin{proposition}
For $d\ge(2g-2)\binom r2$, we have $\sA^{\geq 0}_{r,d}=\sA_{r,d}$.
\end{proposition}

This is proved in \cite[Prop.2.5]{schiffmann_indecomposable}.
We provide below a proof for the comfort of the reader.
 
\begin{lemma}[cf.~Lemma \ref{L:1}]
Let $E$ be an indecomposable vector bundle over $X$.
Then $\si(E)=\set{\nu_1<\dots<\nu_s}$ does not have gaps of length greater than $2g-2$.
\end{lemma}
\begin{proof}
Assume that there is a gap of length greater than $2g-2$, say $\nu_{k+1}-\nu_k>2g-2$. Then there exists an exact sequence
\[0\to E'\to E\to E''\to0,\]
where $E'\in\cA^{\ge \nu_{k+1}}$ and $E''\in\cA^{\le \nu_k}$.
This implies that
\[E''(K)\in\cA^{\le \nu_k+2g-2}\subset\cA^{<\nu_{k+1}}\]
and therefore $\Ext^1(E'',E')\iso\Hom(E',E''(K))^*=0$. 
We conclude that the above sequence
splits and $E$ is not indecomposable.	
\end{proof}

\begin{corollary}
	Assume that $E$ is an indecomposable vector bundle over $X$ of rank $r$
	 and degree $d\ge(2g-2)\binom r2$.
	Then $E\in\cA^{\ge0}$.
\end{corollary}
\begin{proof}
The proof is in all points analogous to the proof of Lemma \ref{L:2}.
\end{proof}

The first formula of the next result was proved by the first author \cite{mozgovoy_motivicb}
in the case of quiver representations. The second formula was proved by the second author \cite{schiffmann_indecomposable}.
We give a unified approach based on \cite{mozgovoy_motivicb}.

\begin{theorem}
\label{th:Om_K}
We have
\begin{align}
\sum_{r,d}\sI^{\geq 0}_K(r,d)\w^r\z^d
&=\sum_{r,d}\sI^{\geq 0}_0(r,d)\w^r\z^d=\Exp\rbr{\frac{q\sum_{r,d}\sA^{\geq 0}_{r,d}\;\w^r\z^d}{q-1}},\\
\sum_{r,d}\sI^{\geq 0,\nil}_{0}(r,d)\w^r\z^d
&=\Exp\rbr{\frac{\sum_{r,d}\sA^{\geq 0}_{r,d}\;\w^r\z^d}{q-1}}.
\end{align}
\end{theorem}
\begin{proof}
To prove the first equation we apply the same approach as in
\cite[Theorem 5.1]{mozgovoy_motivicb}.
The forgetful map
\[\Higgs_K^{\geq 0}(r,d)\to\sCoh^{\geq 0}(r,d)\]
has a fiber over $E\in\sCoh^{\geq 0}(r,d)$ that is equal to
\[\Hom(E,E\ts K)\iso\Ext^1(E,E)^*.\]
If $E=\bop E_\iota^{n_\iota}$ is a decomposition of $E$ into the sum of indecomposable objects
then the contribution of the fiber of $E$ to $\vol([\Higgs^{\geq 0}_K(r,d)(\bk)))$ is equal to (see \cite[Theorem 2.1]{mozgovoy_motivicb})
\[\frac{[\Hom(E,E\ts K)]}{[\Aut E]}
=\frac{[\Ext^1(E,E)]}{[\End(E)]\prod_\iota(q\inv)_{n_\iota}}
=\frac{q^{-\hi(E,E)}}{\prod_\iota(q\inv)_{n_\iota}},
\]
where $(q)_n=(1-q)\dots(1-q^n)$. Note that $\hi(E,E)=r^2(1-g)$.
We conclude from the proof of \cite[Theorem 5.1]{mozgovoy_motivicb} that
\begin{multline*}
\sum_{r,d}\sI^{\geq 0}_K(r,d)w^rz^d
=\sum_{r,d}q^{r^2(1-g)}\vol\rbr{\Higgs^{\geq 0}_K(r,d)(\bk)}w^rz^d\\
=\sum_{n:\Ind\to\bN}\prod_{E\in\Ind}\frac{e^{n(E)\cl E}}{(q\inv)_{n(E)}}
=\Exp\rbr{\frac{\sum \sA^{\geq 0}_{r,d}\;w^rz^d}{1-q\inv}}
\end{multline*}
where we have denoted by $\Ind$ the set of isoclasses of indecomposable objects in $\Coh^{\geq 0}(X)$.

The proof of the second formula goes through the same lines.
Consider the forgetful map
\[\Higgs^{\geq 0,\nil}_{0}(r,d)\to\sCoh^{\geq 0}(r,d).\]
If $E=\bop E_\iota^{n_\iota}$ is a splitting into indecomposables as before
then the contribution of $E$ in $\vol(\Higgs^{\geq 0,\nil}_{0}(r,d)(\bk))$
is equal to \cite[Cor.2.4]{schiffmann_indecomposable}
\[\frac{[\Hom^\nil(E,E)]}{[\Aut E]}=\prod_{\iota}\frac{q^{-n_\iota}}{(q\inv)_{n_\iota}}.\]

Applying again the proof of \cite[Theorem 5.1]{mozgovoy_motivicb} we conclude that
\begin{multline*}
\sum_{r,d}\sI^{\geq 0, \nil}_{0}(r,d)w^rz^d
=\sum_{r,d}\vol(\Higgs^{\geq 0,\nil}_{0}(r,d)(\bk))w^rz^d\\
=\sum_{n:\Ind\to\bN}\prod_{E\in\Ind}\frac{q^{-n(E)}e^{n(E)\cl E}}{(q\inv)_{n(E)}}
=\Exp\rbr{\frac{\sum q\inv\sA^{\geq 0}_{r,d}\;w^rz^d}{1-q\inv}}.
\end{multline*}
\end{proof}

\begin{corollary}
\label{cor:K0}
	We have, for any pair $(r,d)$,
\begin{enumerate}
\item $\Om_K^{\geq 0}(r,d)=q\Om_{0}^{\geq 0,\nil}(r,d)=q\sA^{\geq 0}_{r,d}$,
\item $\Om_K(r,d)=q\sA_{r,d}$.
\end{enumerate}
\end{corollary}
\begin{proof}
The first statement follows from Theorem~\ref{th:Om_K} and the definition of the DT-invariants $\Om_D^{\ge0}(r,d)$ and $\Om_D^{\ge0,\nil}(r,d)$.
We prove the second.
If $d\ge(2g-2)\binom r2$ then $\Om_K(r,d)=\Om_K^{\geq 0}(r,d)$,
$\sA_{r,d}=\sA^{\geq 0}_{r,d}$,
hence $\Om_K(r,d)=q\sA_{r,d}$ by the first statement.
For arbitrary $r,d$ we note that
$\Om_K(r,d)=\Om_K(r,d+r)$ and $\sA_{r,d}=\sA_{r,d+r}$.
\end{proof}

\vsp{.1in}

Let us now turn to the case of quiver sheaves. We do not know of a formula similar to those of Theorem~\ref{th:Om_K} expressing the volume of $\QS_0^{\geq 0}(\ur,\ud)$ or $\QS_0^{\nil}(\ur,\ud)$ in terms of Kac polynomials $\sA_{r,d}$. However, one still has the following relation between the volumes of $\QS_0^{\geq 0}(\ur,\ud)$
and $\QS_0^{\geq 0,\nil}(\ur,\ud)$.

\vsp{.1in}

%

\begin{proposition}\label{prop:qsnil}
We have the following equality of formal series in $\mathbb{T}$:
$$\sum_{\ur,\ud} \sI^{\geq 0}_{0}(\ur,\ud)w^{\ur}z^{\ud}
=\left(\sum_{\ur,\ud} \sI^{\geq 0,\nil}_{0}(\ur,\ud)w^{\ur}z^{\ud}\right) \cdot \Exp\rbr{\sum_{r,d} \sA^{\geq 0}_{r,d}\;w^{r\delta}z^{d\delta }}$$
where $\delta=(1,1,\ldots, 1) \in \mathbb{Z}^I$.
\end{proposition}
\begin{proof}
Note that the subalgebra $\bigoplus_{r,d} \mathbb{Q}(q^{\frac{1}{2}})w^{r\delta}z^{d\delta}$ of $\mathbb{T}$ is commutative hence the plethystic exponential is well-defined.
 Let $\cA_0^{\geq 0,\isom}$ be the full subcategory of $\cA^{\geq 0}_0$ consisting of quiver sheaves
$\oE=(E_i,\theta_i)_i$ for which $\theta_i:E_i \simeq E_{i+1}$ for all $i$. We claim that any object $\oE \in \cA^{\geq 0}_0$ has a unique subobject $\oE'$ satisfying
$$\oE' \in \cA_0^{\geq 0, \isom}, \qquad \oE/\oE'\in \cA_0^{\geq 0, \nil}.$$
To see this, consider the decreasing filtration 
$\oE\sps \theta(\oE)\sps \theta^2(\oE) \dots$.
Since $\End(\bigoplus_i E_i)$ is finite-dimensional, this filtration stabilizes and we let $\oE'$ denote its limit. By construction and because $\Coh^{\geq 0}(X)$ is stable under taking quotients, $\oE' \in \cA_0^{\geq 0, \isom}$ and $\oE/\oE' \in \cA_0^{\geq 0, \nil}$. This shows the existence of a filtration of the desired form. Unicity comes from the easily checked fact that $\Hom(\oE',\oE'')=\{0\}$ whenever $\oE' \in \cA^{\isom}$ and $\oE''\in \cA_0^{\nil}$. Setting $\oE''=\Ker(\theta^n)$ for $n \gg 0$ yields in fact a canonical splitting of the exact sequence $0 \to \oE' \to \oE \to \oE/\oE' \to 0$ but we won't need this.
Put $\ga=(\ur,\ud)$ and
$$\sI^{\geq 0,\isom}_0(\ur,\ud)
=\sum_{\oE\in\cA^{\geq 0,\isom}_0(\ur,\ud)}(-q^\oh)^{\hi_0(\ga,\ga)}\frac{1}{\n{\Aut\oE}}.$$
From the unicity of the filtration $\oE' \subseteq \oE$ above we have by a standard argument in the Hall algebra
$$\sum_{\ur,\ud}\sI^{\geq 0}_{0}(\ur,\ud)w^{\ur}z^{\ud}=\left(\sum_{\ur,\ud}\sI^{\geq 0,\nil}_{0}(\ur,\ud)w^{\ur}z^{\ud}\right)\cdot \left(\sum_{\ur,\ud}\sI^{\geq 0,\isom}_{0}(\ur,\ud)w^{\ur}z^{\ud}\right).$$
Observe that $\sI^{\geq 0, \isom}_0(\ur,\ud) =0$ unless $(\ur,\ud)=(r\delta,d\delta)$ for some $(r,d)$. All that remains to prove is the following equality:
\begin{equation}\label{E:nilisolast}
\sum_{r,d}\sI^{\geq 0,\isom}_{0}(r\delta,d\delta)w^{r\delta}z^{d\delta}= \Exp\rbr{\sum_{r,d} \sA^{\geq 0}_{r,d}\;w^{r\delta}z^{d\delta }}.
\end{equation}
The proof of that last statement is of a similar nature to that of Theorem~\ref{th:Om_K}. Let $\QS^{\geq 0, \isom}_0(r\delta,d\delta)$ be the stack parametrizing objects in $\cA^{\geq 0, \isom}_0$ of class $(r\delta,d\delta)$.
 Consider the forgetful map $\pi:\QS^{\geq 0,\isom}_0(r\delta,d\delta) \to \sCoh^{\geq 0}(r,d)$. For any
positive coherent sheaf $E \in \sCoh^{\geq 0}(r,d)$, the fiber of $\pi$ contributes a volume of
$\prod_i \n{\Aut E} / \prod_i \n{\Aut E}=1$. It follows that 
\begin{equation}\label{E:froup}
\vol(\QS^{\geq 0, \isom}_0(r\delta,d\delta)(\bk))
=\n{\set{E \in \Coh^{\geq 0}(X)\dv \cl E=(r,d)} / \sim}.
\end{equation}
Let us denote by $m_{r,d}$ the r.h.s.\ of (\ref{E:froup}). The equality (\ref{E:nilisolast}) is now a consequence of the next lemma (cf.~\cite[Lemma 5]{mozgovoy_computational}).
\end{proof}

\begin{lemma} We have $\sum_{r,d} m_{r,d}w^rz^d=\Exp\left( \sum_{r,d} \sA^{\geq 0}_{r,d}\;w^rz^d\right)$.
\end{lemma}
\begin{proof}
The proof is close to that of \cite[Theorem 5.1]{mozgovoy_motivicb}
or of \cite[Proposition. 2.2]{schiffmann_indecomposable}. For any $l \in \mathbb{N}$, let us denote by $\Ind^{\geq 0}_{(r,d),l}$ the set of isoclasses of indecomposable positive coherent sheaves $E$ on $X$ of class $(r,d)$ for which $E \otimes_{\bk}\overline{\bk}$ splits as a direct sum of $l$ geometrically indecomposable coherent sheaves. Note that
$\Ind^{\geq 0}_{(r,d),l}$ is empty unless $l \dv \gcd(r,d)$, see \cite[Lemma~2.6]{schiffmann_indecomposable}. By \cite[(2.4), (2.5)]{schiffmann_indecomposable} we have, for every $n, r,d$
$$\sA^{\geq 0}_{X_n,r,d}=\sum_{l\dv n}l \n{\Ind^{\geq 0}_{(lr,ld),l}}$$
and
$$\sum_{l \geq 1}\sum_{r,d} \frac{1}{l} \sA^{\geq 0}_{X_l,r,d}\;w^{lr}z^{ld}=\sum_{l \geq 1}\sum_{r,d} \frac{1}{l} \n{\Ind^{\geq 0}_{r,d}}w^{lr}z^{ld}.$$
We deduce that
\begin{multline*}
\Exp\left(\sum_{r,d} \sA^{\geq 0}_{X,r,d}\;w^{r}z^{d}\right)=\exp\left(\sum_{l \geq 1}\frac{1}{l}\sA_{X_l,r,d}\;w^{lr}z^{ld}\right)=\prod_{r,d}\exp\left(\sum_{r,d} \n{\Ind^{\geq 0}_{r,d}}\;w^{lr}z^{ld}\right)\\
=\prod_{r,d} \frac{1}{(1-w^rz^d)}\n{\Ind^{\geq 0}_{r,d}}=\sum_{r,d}m_{r,d}w^rz^d,
\end{multline*}
as wanted.
\end{proof}

\vsp{.2in}

\section{Counting nilpotent quiver sheaves}
\label{sec: nilp}
\vspace{.1in}

The purpose of this section is to give an explicit formula counting the nilpotent quiver sheaves (of fixed rank and degree) 
which belong to $\cA_D^{\ge0}$, \textit{under the assumption that $l\leq0$}.
As in \cite{schiffmann_indecomposable} (in the special case $D=0$), we first stratify the collection of such nilpotent quiver sheaves according to some Jordan type, and then reduce the computation of the count for each strata to the computation of some truncated Eisenstein series.

\vspace{.1in}

\subsection{Jordan stratification}
\label{sec:jordan}
We do not assume that $l \leq 0$ here.
Let $\oE=(E,\te)\in\cA_D$.
For any $k\ge0$, define $\te^k$ to be the composition
$$E\to E[1]\to\dots\to E[k]$$
and set $F_k=\im\te^k[-k]\sbs E$.
Assume that $(E,\te)$ is a nilpotent quiver sheaf, of nilpotency index~$s$.
By construction we have a chain of inclusions
$$E=F_{0}\hlr F_{1}\hlr F_{2}\hlr\cdots\hlr F_{s}=0$$
and a chain of epimorphisms 
$$E=F_{0}\epi F_{1}[1]\epi F_{2}[2]\epi\cdots\epi F_{s}[s]=0.$$
Let us set
$$F'_{k}=\ker(F_{k}\to F_{k+1}[1]), \qquad F''_{k}=\coker(F_{k+1}\to F_{k}).$$
Then we have a chain of inclusions
$$F'_{0}\hlr F'_{1}\hlr F'_{2}\hlr\cdots\hlr F'_{s}=0$$
and a chain of epimorphisms
$$F''_{0}\epi F''_{1}[1]\epi F''_{2}[2]\epi\cdots\epi F''_{s}[s]=0.  $$
Let us finally set
$$\al_{k}=\cl F''_{k}[k]-\cl F''_{k+1}[k+1]\in\Ga,
\qquad k\ge0.$$

\begin{lemma}\label{L:Formulas1}
The following hold:
$$\cl F''_{k}=\sum_{j \ge k} \al_{j}[-k], \qquad
\cl F'_{k}=\sum_{j \ge k} \al_{j}[-j].$$
\end{lemma}
\begin{proof}The first statement is immediate from the definition of $\al_{k}$.
The second statement is then a consequence
of the relations
$$\cl F_{k}=\sum_{j \ge k}\cl F''_{j}$$
and
$$\cl F'_{k}=\cl F_{k}-\cl F_{k+1}[1].$$
\end{proof}

\vspace{.1in}

We will call the tuple $\ual=(\al_{k})_{k}$ the \textit{Jordan type} of $\oE=(E,\te)$.
For convenience, we will write 
\eq{f''_{k}(\ual)=\sum_{j \ge k} \al_{j}[-k]\in\Ga, \qquad f'_{k}(\ual)=\sum_{j \ge k} \al_{j}[-j]\in\Ga}
and set $|\ual|=\sum_k f''_{k}\in\Ga$.
Note that $\cl\oE=|\ual|$.

\vspace{.1in}

For any fixed tuple $\ual=(\al_{k})_{k}$, we write $\Nil_D(\ual)=\QS_D^\nil(\ual)$ for the locally closed substack of $\QS_D^\nil(\n\ual)$ whose objects are nilpotent quiver sheaves of Jordan type $\ual$.
Intersecting it with $\QS_D^{\ge0,\nil}{\n\ual}$ yields an open substack $\Nil_D^{\ge0}(\ual)=\QS_D^{\ge0,\nil}(\ual)$.

\subsection{The forgetful map}
For $\ual=(\al_0,\al_1,\dots,\al_{s-1})$ a tuple of elements of $\Ga$, we let $\Filt(\ual)$
denote the stack of chains of epimorphisms in $\cA$
$$E_0\epi E_1\epi\cdots\epi E_{s}=0$$
such that 
$$\cl\ker(E_k\epi E_{k+1})=\al_k \qquad \forall \;k=0,\dots,s-1.$$
We denote by
$\Filt^{\ge0}({\ual})$ the open susbstack of $\Filt({\ual})$ consisting of
chains $E_0\epi E_1\epi\cdots $ such that $E_0\in\cA^{\ge0}$.
We use notation $\fCoh(\ual)=\Filt(\ual)$ and $\fCoh^{\ge0}(\ual)=\Filt^{\ge0}(\ual)$ for $n=1$.

Consider the map (see~\S\ref{sec:jordan} for notation)
\eq{\label{eq:varpi}
\varpi_{\ual}:\Nil_D(\ual)\to\Filt(\ual),\qquad
(E,\te) \mto (F''_{0}\epi F''_{1}[1]\epi\cdots\epi F''_{s}[s]=0).}
From the fact that the category $\cA^{\ge0}$ is closed under taking quotients, it follows that $\varpi_{\ual}$
restricts to a map
$\Nil_D^{\ge0}(\ual)\to \Filt^{\ge0}({\ual})$.

\begin{proposition}\label{P:31}
The volume of the fiber of the map $\varpi_{\ual}$ over any object of
$\Filt({\ual})(\bk)$  is equal to 
$$\prod_{k \ge0}q^{-\chi( f''_{k}(\ual), f'_{k+1}(\ual))}.$$
\end{proposition}
\begin{proof}
Let \cT be the category of triples $(F^{(1)},F^{(2)},\te)$, where $F^{(1)},F^{(2)}\in\cA$ and $\te:F^{(1)}\to F^{(2)}[1]$.
Given a nilpotent quiver sheaf $(E,\te)$, we can define objects $\oF_k=(F_k,F_{k+1},\te)\in\cT$, for $k\ge0$, together with monomorphisms
$$\oF_0\hlr\oF_1\hlr\oF_2\hlr\dots$$
By the discussion in the previous section, the category of nilpotent quiver sheaves of Jordan type $\ual$ is equivalent to the category $\cD$ consisting of tuples $(\oF_k\in\cT)_{k=0,\dots,s}$
equipped with a chain of monomorphisms 
$$\oF_0\hlr\oF_1\hlr\oF_2\hlr\dots,$$
isomorphisms $F^{(1)}_{k+1}=F^{(2)}_{k}$ for all $k$,
and satisfying $\cl F^{(1)}_{k}=\sum_{j\ge k} f''_{j}(\ual)$ for all $k$.
Under the equivalence $\Nil_D(\ual)(\bk) \simeq \ang{\cD}$ the map $\varpi_{\ual}$ is given by the functor 
$$\ang{\cD}\to\Filt(\ual)(\bk),\qquad
(\oF_k)_k\mto\rbr{
\oF_0^{(1)}/\oF_1^{(1)}\epi
(\oF_1^{(1)}/\oF_2^{(1)})[1]\epi\dots\epi
\oF_{s}^{(1)}[s]=0
}$$
Let $\bar H=(H_{0}\epi H_{1}\epi\cdots)$ be an object of
$\Filt({\ual})(\bk)$.
Let $F''_{k}=H_{k}[-k]$,
so that $\cl F''_{k}=f''_{k}(\ual)$.
Define
$$\oF''_k=(F''_{k}, F''_{k+1}, \theta)\in\cT,$$
where $\theta: F''_{k}\epi F''_{k+1}[1]$ is induced by the map $H_{k}[-k]\epi H_{k+1}[-k]$.
By construction, an object of the fiber of $\bar H$ corresponds to an iterated extension, in the category $\mathcal{T}$ of the objects $\oF''_k$. More precisely, we may canonically reconstruct objects $\oE$ of the fiber of $\bar H$ as follows: we inductively build
exact sequences in $\mathcal{T}$
\eq{\label{E:extt}
0\to\oF_{k+1}\to\oF_{k}\to\oF''_{k}\to0}
together with identifications
$F^{(2)}_{k}=F^{(1)}_{k+1}=:F_{k+1}$
$$\begin{tikzcd}
0\rar&F_{k+1}\rar[dashed]\dar{}&F_k\rar[dashed]\dar[dashed]&F_{k}''\rar\dar{\te}&0\\
0\rar&F_{k+2}[1]\rar&F_{k+1}[1]\rar&F_{k+1}''[1]\rar&0
\end{tikzcd}$$
starting from $\oF_{s}=\oF''_s=0$
and letting $k=s-1, \ldots ,0$; 
we then set $\oE=(F^{(1)}_{0}, \theta)$,
where $\te$ is the composition 
$F^{(1)}_{0}\to F^{(2)}_{0}[1] \simeq F^{(1)}_{1}[1] \hookrightarrow F^{(1)}_{0}[1]$. 

\vspace{.1in}

In order to keep track of these successive extensions, we will use the following result.
Let $\oE=(E^{(1)}, E^{(2)},\phi)$, $\oF=(F^{(1)}, F^{(2)}, \psi)$ be a pair of objects of $\cT$.
Consider the groupoid $\cC$ whose objects are short exact sequences
\eq{0\to\oF\to\oG\to\oE\to0}
in $\mathcal{T}$ and the groupoid $ \mathcal{C}'$ whose objects are short exact sequences
\eq{0\to F^{(2)}\to G\to E^{(2)}\to0.}
The set of isoclasses of objects in $\mathcal{C}$ is $\Ext^1_{\cT}(\oE,\oF)$ and,
for any $\eta\in \Ext^1_{\cT}(\oE,\oF)$, we have $\Aut(\eta)=\Hom_{\cT}(\oE,\oF)$.
Likewise, the set of isoclasses
of objects in $\mathcal{C}'$ is $\Ext^1_\cA(E^{(2)}, F^{(2)})$ and, for any $\ga\in\Ext^1_\cA(E^{(2)}, F^{(2)})$,
we have $\Aut(\ga)=\Hom_\cA(E^{(2)},F^{(2)})$.
There is an obvious forgetful functor $\Phi: \mathcal{C}\to \mathcal{C}'$.

\begin{lemma}\label{L:extt}
Assume that $\psi:F^{(1)}\to F^{(2)}[1]$ is an epimorphism. Then the orbifold volume of any fiber
of $\Phi:\cC\to\cC'$ is equal to $q^{-\chi (\oE,\oF) +\chi( E^{(2)}, F^{(2)})}$.
\end{lemma}
\begin{proof}
By \cite{gothen_homological} there is a long exact sequence
\begin{multline}\label{E:les}
\Ext^1_{\mathcal{T}}(\oE, \oF)\to \Ext^1_\cA(E^{(1)},F^{(1)}) \oplus \Ext^1_\cA(E^{(2)}, F^{(2)})\to
\Ext^1_\cA(E^{(1)},F^{(2)}[1])\to \\ 
\to \Ext^2_\cT(\oE,\oF)\to 0
\end{multline}
Because $\psi$ is an epimorphism, the map 
$\Ext^1(E^{(1)},F^{(1)})\to\Ext^1(E^{(1)},F^{(2)}[1])$ is onto by Serre duality.
It follows that $\Ext^2(\oE,\oF)=0$ and that the composed map
$$\Ext^1(\oE, \oF)\to\Ext^1(E^{(1)},F^{(1)}) \oplus \Ext^1(E^{(2)}, F^{(2)})\to
\Ext^1(E^{(2)},F^{(2)})$$
is surjective.
Therefore the functor $\Phi$ is essentialy surjective on objects and the set of isoclasses of
objects $\Phi^{-1}(\ga)$ is of cardinality 
$q^{\dim\Ext^1(\oE,\oF)-\dim\Ext^1(E^{(2)},F^{(2)})}$.
Taking into account the automorphisms of objects and using the fact that
$\chi(\oE,\oF)=\dim\Hom(\oE,\oF)-\dim\Ext^1(\oE,\oF)$
yields the statement of the lemma.
\end{proof}

\vspace{.1in}

We may now finish the proof of Proposition~\ref{P:31}. Starting from $\oF_s=\oF''_s$, we inductively build objects $\oF_k\in\cT$ and exact sequences
(\ref{E:extt}) in such a way that
$F^{(2)}_{k}=F^{(1)}_{k+1}=:F_{k+1}$ for all $k$.
We obtain inductively that the maps $F_k^{(1)}\to F_{k}^{(2)}[1]$ are epimorphisms.
By Lemma~\ref{L:extt}, each step contributes a factor of $q^{-\chi( \oF''_k, \oF_{k+1}) + \chi(F''_{k+1}, F_{k+2})}$ to the volume of the fiber.
It remains to observe that because of the exact sequence (\ref{E:les}), we have
\begin{equation*}
\begin{split}
-\chi(\oF''_k, \oF_{k+1}) +\chi( F''_{k+1}, F^{(2)}_{k+1}) &=-\chi(F''_{k}, F^{(1)}_{k+1})
+\chi(F''_{k},F^{(2)}_{k+1}[1])\\
&=-\chi(F''_{k},F^{}_{k+1}-F^{}_{k+2}[1])\\
&=-\chi(F''_{k},F'_{k+1}).
\end{split}
\end{equation*}
\end{proof}

\vspace{.1in}

From the formulas in Proposition~\ref{P:31} and Lemma~\ref{L:Formulas1} one finds that the volume of each fiber of $\varpi_{\ual}:\Nil_D(\ual)\to\Filt(\ual)$ is the same as that of an affine space of dimension equal to
\begin{equation}\label{E:fibersforget}
-\sum_{k \ge0} \chi( f''_{k}(\ual), f'_{k+1}(\ual))
=-\sum_{k \ge0} \sum_{\substack{l_1 \ge k \\ l_2 \ge k+1}} 
\chi(\al_{l_1}[-k],\al_{l_2}[-l_2])=:a_D(\ual).
\end{equation}

The map (cf.~\S\ref{sec:notation stacks})
\eq{\label{eq:stack vb}
\Filt(\ual)\to \prod_k \stA({\al_k}),\qquad
(E_0\epi E_1\epi\cdots\epi E_s=0)\mto (\ker(E_k\epi E_{k+1}))_k}
is a stack vector bundle of rank 
$-\sum_{j>k} \chi (\al_j, \al_k)$
(see \cite[\S3.1]{garcia-prada_motives}).
In particular, $\Filt({\ual})$ is smooth and 
\eq{\label{eq:vol chain}
\vol(\Filt({\ual})(\bk))=q^{-\sum_{j>k}\chi(\al_j,\al_k)}\prod_k \vol(\stA({\al_k})(\bk)).}
We obtain from \eqref{eq:vol chain} and Proposition \ref{P:31} that
\begin{equation}\label{E:volJordan}
\vol(\Nil_D(\ual)(\bk))
=q^{a_D(\ual)-\sum_{j>k}\chi(\al_{j},\al_{k})}\prod_{k} \vol(\stA({\al_{k}})(\bk)).
\end{equation}

\subsection{Volume of stacks of positive nilpotent quiver sheaves.}
We assume that $l \leq 0$. Fix $\al\in \Ga=(\mathbb{Z}^2)^I$ such that $\mu(\al)\ge0$. There are only finitely many $\ual $ satisfying $|\ual|=\al$ for which $\Nil^{\ge0}_D(\ual)$ is not empty; indeed, there are finitely many
possible choices for $f''_{i,k}$ satisfying $\mu(f''_{i,k}) \ge0$ and $\sum_k f''_{i,k}=\al_i$.
 
 \vspace{.1in}

\begin{proposition}\label{P:32} Assume that $l \leq 0$. Then the following diagram is cartesian
$$\begin{tikzcd}
\Nil_D^{\ge0}(\ual)\rar\dar["\varpi_{\ual}"']
&\Nil_D(\ual)\dar["\varpi_{\ual}"]\\
\Filt^{\ge0}(\ual)\rar &\Filt(\ual)
\end{tikzcd}$$
where the horizontal arrows stand for the open immersions.
\end{proposition}
\begin{proof}
We must show that $(E,\te)\in\Nil_D(\ual)$ belongs to $\Nil_D^{\ge0}(\ual)$ if and only if $F''_{0}$
belongs to $\cA^{\ge0}$.
Let us first assume that $(E,\te)$ belongs to $\Nil_D^{\ge0}(\ual)$.
As $\cA^{\ge0}$ is closed under taking quotients and $F''_{0}$ is a quotient of $F_{0}=E$
we have $F''_{0}\in\cA^{\ge0}$.
Conversely, assume that $F''_{0}\in\cA^{\ge0}$.
Then by the same argument, $F''_{k}[k]\in\cA^{\ge0}$ and hence $F''_{k}\in\cA^{\ge0}$, for all $k$, since $D$ is negative.
But since $E=F_0$ is a successive extension
of objects $F''_{0}, F''_{1},\dots$ and since $\cA^{\ge0}$ is stable under extensions, we deduce that $E$ belongs to $\cA^{\ge0}$ as well. We are done.
\end{proof}

As an immediate corollary of Propositions~\ref{P:31} and ~\ref{P:32} we obtain the following formula:

\begin{corollary}
\label{cor:vol of strata}
Assume that $l \leq 0$. Then the volume of the stack $\Nil_D^{\ge0}(\ual)$ is equal to
$$\vol\rbr{\Nil_D^{\ge0}({\ual})(\bk)}
=\sum_{|\ual|=\al} q^{a_D(\ual)} \vol\rbr{\Filt^{\ge0}({\ual})(\bk)},$$
where $a_D(\ual)$ is defined as the r.h.s of (\ref{E:fibersforget}).
\end{corollary}

The volumes of the stacks $\Filt^{\ge0}(\ual)=\prod_{i\in I}\fCoh^{\ge0}(\ual_i)$ have been explicitly computed in \cite{schiffmann_indecomposable}. This yields a closed (albeit complicated) formula for the volumes of all the stacks
$\Nil_D^{\ge0}(\ual)$. 


\vspace{.2in}

\section{Computation of DT invariants -- the Higgs case}
\label{sec6}
In this section we use the results of Sections~\ref{sec3} and \ref{sec4} to derive a closed formula for the volume of the stacks $\QS^{ss}_D(\alpha)$
when $n=1$ and $l\ge 2g-2$, i.e.\ when the moduli stack in
question is the moduli stack of semistable meromorphic Higgs bundles associated to a divisor $D$.
Note that the case $l=2g-2$ is covered by Corollary~\ref{cor:K0} and  \cite{schiffmann_indecomposable}.

\vsp{.1in}

Assume that $l=\deg D\le 0$.
We first observe that when $n=1$ we may associate a partition $\lambda(\ual)$ to any Jordan type $\ual$ by setting $\lambda(\ual)=(1^{r_1}, 2^{r_2}, \ldots)$, where $\al_i=(r_i,d_i)$. 
We then have (cf.~\eqref{E:fibersforget})
$$a_D(\ual)=a_0(\ual)
+\frac{l}{2}\rbr{\sum_i i r_i}^2-\frac{l}{2}\sum_{k\geq 1} \rbr{\sum_{i\geq k}r_i}^2
=a_0(\ual) + \frac{l}{2} r^2-\frac{l}{2}
\ang{\la(\ual),\la(\ual)}.$$

\vsp{.05in}

Let $J(r,d)$ stand for the set of all tuples $\ual=(\al_1, \ldots, \al_s)$ such that $\sum i \al_i=(r,d)$ and $\al_s \ne 0$, and let $J_\gen(r)$ stand for the set of all sequences $\ur=(r_1, \ldots, r_t)$ such that $\sum_i ir_i=r$ and $r_t\geq 1$.
There is a natural map
$\pi:J(r,d) \to J_\gen(r)$ which assigns to a tuple
$(\al_1, \ldots, \al_s)$ the sequence $(\rk \al_1, \ldots, \rk \al_s)$ in which all the \textit{last} zero entries have been removed. 
Let us set
$$\sI^{\geq 0,\nil}_D(\ual)=(-q^{\frac{1}{2}})^{-lr(\ual)^2}\vol(\Nil_D^{\geq 0})(\ual)(\bk)$$
where $r(\ual)=\sum_i i r_i$,
so that by Propositions \ref{P:32} and \ref{P:31} we have
$$\sI^{\geq 0,\nil}_{D}(r,d)
=\sum_{\ual \in J(r,d)} \sI^{\geq 0,\nil}_{D}(\ual)
=(-q^\oh)^{-lr^2} \sum_{\ual \in J(r,d)} q^{a_0(\ual)+\frac l2 r^2-\frac l2\langle \lambda(\ual), \lambda(\ual)\rangle}
\vol\rbr{\fCoh^{\geq 0}({\ual_i})(\bk)}.$$
Let us fix some $\ur=(r_1, \ldots, r_t) \in J_\gen(r)$ and put $\lambda=(1^{r_1}, 2^{r_2}, \ldots,t^{r_t})$. Using \cite[Sec.5.6]{schiffmann_indecomposable} we have
\begin{equation*}
\begin{split}
\sum_{\ual \in \pi^{-1}(\ur)} \sI^{\geq 0,\nil}_{D}(\ual)\z^{\sum_i i \deg \al_i}&
=(-q^{\frac{1}{2}})^{-l\ang{\la,\la}}\sum_{\ual \in \pi^{-1}(\ur)} q^{a_0(\ual)}
\vol\rbr{\fCoh^{\geq 0}({\ual})(\bk)}
\z^{\sum_i i \deg \al_i}\\
&=(-q^{\frac{1}{2}})^{-l\langle \lambda,\lambda\rangle} q^{(g-1)\langle \lambda,\lambda\rangle} J_{\lambda}(\z)H_{\lambda}(\z) \cdot 
\Exp\rbr{\frac{|X(\bk)|}{q-1}\cdot\frac{\z}{1-\z}}.
\end{split}
\end{equation*}
Summing over $\ur \in J_\gen(r)$ and over all $r$, we obtain the following formula:

\begin{equation}\label{E:sumgenJor}
\sum_{r,d} \sI^{\geq 0,\nil}_D(r,d) w^rz^d=\sum_{\lambda}(-q^{\frac{1}{2}})^{(2g-2-l)\langle \lambda,\lambda\rangle} J_{\lambda}(\z)H_{\lambda}(\z) \cdot \Exp\left( \frac{|X(\bk)|}{q-1} \cdot \frac{\z}{1-\z}\right).
\end{equation}

Let $\QS^{\geq 0, \nil}_{D,\bun}(r,d)$ denote the open substack of $\QS^{\ge0,\nil}_{D}(r,d)$ whose objects are vector bundles, and set
$\sI^{\geq 0, \nil}_{D,\bun}(r,d)
=(-q^\oh)^{-lr^2}\vol\rbr{\QS^{\geq 0, \nil}_{D,\bun}(r,d)(\bk)}$.
 
 \vsp{.1in}
 
 \begin{lemma} The following hold:
 \begin{enumerate}
 \item
 \begin{equation*}
\sum_{r, d} \sI^{\geq 0,\nil}_{D}(r,d)\w^r\z^d
=\sum_{r,d}\sI^{\geq 0,\nil}_{D,\bun}(r,d)\w^r\z^d \cdot
\sum_{d\ge0}\sI^{\geq 0,\nil}_{D}(0,d)\z^d.
\end{equation*}
\item
\[\sum_{d\ge0}\sI^{\geq 0,\nil}_{D}(0,d)\z^d
=\Exp\rbr{\frac{|X(\bk)|}{q-1}\sum_{d\ge1}\z^d}
=\Exp\rbr{\frac{|X(\bk)|}{q-1}\cdot\frac{\z}{1-\z}}.\]
\end{enumerate}
 \end{lemma}
 \begin{proof}
Let $F \in \Coh^{\geq 0}(X)$ be a coherent sheaf, and $F=V \oplus T$ be a decomposition as a direct sum of a vector bundle $V$ and a torsion sheaf $T$. Observe that $V$ and $T$ belong to $\Coh^{\geq 0}(X)$. We have
$$\Hom(F,F(D)) = \Hom(V,V(D)) \oplus \Hom(T,T(D)) \oplus \Hom( V,T(D))$$
and $\te \in \Hom(F,F(D))$ is nilpotent if and only if its projections to $\Hom(V,V(D))$ and $\Hom(T,T(D))$ are. On the other hand there is a canonical exact sequence
\begin{equation*}
1\to\Hom(V,T)\to\Aut F\to\Aut V\times \Aut T\to1.
\end{equation*}
We deduce that
$$\frac{\n{\Hom^{\nil}(F,F(D))}}{\n{\Aut F}}=\frac{\n{\Hom^{\nil}(V,V(D))}}{\n{\Aut V}}\cdot \frac{\n{\Hom^{\nil}(T,T(D))}}{\n{\Aut T}}.$$
Equation i) readily follows. Statement ii) is proved as the second equality of Theorem~\ref{th:Om_K}; observe that the number of absolutely indecomposable torsion sheaves of degree $d>0$ is $\n{X(\bk)}$), hence
\[\sum_{d\ge0}\sI^{\geq 0,\nil}_{D}(0,d)\z^d
=\Exp\rbr{\frac{|X(\bk)|}{q-1}\sum_{d\ge1}\z^d}
=\Exp\rbr{\frac{|X(\bk)|}{q-1}\cdot\frac{\z}{1-\z}}.\]
\end{proof}

The above lemma, together with equation \eqref{E:sumgenJor} implies that, for $\deg D\le0$,
$$
\sum_{r,d} \sI^{\ge0,\nil}_{D,\bun}(r,d)\w^r\z^d
=\sum_\la(-q^\oh)^{(2g-2-l)\ang{\la,\la}}
J_{\lambda}(\z)H_{\lambda}(\z) \w^{\n\la}.
$$
Therefore, for $l=\deg(D) >2g-2$, by Corollary~\ref{C:SDQS}
$$
\sum_{r,d} \sI^{\geq 0}_{D,\bun}(r,d)\w^r\z^{d}
=\sum_{r,d} \sI^{\ge0,\nil}_{K-D,\bun}(r,d)\w^r\z^d
=\sum_\la(-q^\oh)^{l\ang{\la,\la}}
J_{\lambda}(\z)H_{\lambda}(\z) \w^{\n\la}.
$$
Since any semistable Higgs pair of positive rank is a vector bundle, we have
\eq{
\sum_{r>0, d}\frac{\Om^{\geq 0}_D(r,d)}{q-1}\w^r\z^d
=\Log\rbr{\sum_{r,d}\sI^{\geq 0}_{D,\bun}(r,d)\w^r\z^d}
=\Log\rbr{\sum_\la(-q^\oh)^{l\ang{\la,\la}}
J_{\lambda}(\z)H_{\lambda}(\z) \w^{\n\la}}.}

By Proposition~\ref{P:1} we have $\Omega_D^{\geq 0}(r,d)=\Omega_D(r,d)$ for large large enough values of $d$ (depending on $r$), hence $\Om_D^{\ge0}(r,d)$ is $r$-periodic for $d \gg 0$. This implies that the rational function $X_r(z)$ defined in Theorem~\ref{T:Maain} is regular outside of $r$-th roots of unity and has at most simple poles. In addition, for large enough $N$,
$$\Om_D(r,d)=\Om_D^{\geq 0}(r,d+Nr)=-\sum_{\xi \in \mu_r}\xi^{-d}\Res_{z=\xi}\br{X_r(z) \frac{dz}{z}}$$
where $\mu_r$ stands for the set of $r$-th roots of unity.

If $D=K$, then by Corollary \ref{cor:K0}
\begin{multline}
\sum_{r>0, d}\frac{q\inv\Om^{\ge0}_K(r,d)}{q-1}\w^r\z^d
=\Log\rbr{\sum_{r,d}\sI^{\geq 0,\nil}_{0,\bun}(r,d)\w^r\z^d}\\
=\Log\rbr{\sum_\la(-q^\oh)^{(2g-2)\ang{\la,\la}}
J_{\lambda}(\z)H_{\lambda}(\z) \w^{\n\la}}.
\end{multline}
By the same argument as above, for large enough $N$,
$$q\inv\Om_K(r,d)=q\inv\Om_K^{\geq 0}(r,d+Nr)=-\sum_{\xi \in \mu_r}\xi^{-d}\Res_{z=\xi}\br{X_r(z) \frac{dz}{z}}.$$
Theorem~\ref{T:Maain} is proved.

\vsp{.2in}

\section{Harder-Narasimhan recursion -- the general quiver sheaf case}
\label{sec:HN recursion}
\vsp{.1in}

As mentioned above, the nonvanishing of the Euler form on the category of twisted quiver sheaves prevents us from
using the standard DT machinery (involving plethystic logarithms and exponentials) to compute explicitly the Poincar\'e 
polynomial of the moduli stacks of stable quiver sheaves. Nevertheless, these Poincar\'e polynomials are uniquely determined
by the knowledge of the collection of volumes $\sI^{\geq 0}_{D}(\alpha)$ for all $\alpha$, as the following general (and certainly well-known) result shows.

\vsp{.1in}
We consider the following data:
\begin{enumerate}
\item a commutative ring $R$ and an invertible element $t \in R^*$,
\item an $R$-valued skew-symmetric form $\ang{-,-}$ on a lattice $\Ga$ and two linear forms $\br,\bd : \Ga\to\bZ$,
\item a strictly convex cone $C\sbs\Ga_\bR=\Ga\ts_\bZ\bR$ (that is, $C$ is a cone and $f(C) \subsetneq \mathbb{R}$ for any linear form $f \in \Ga_\bR^*$), 
\item a map $a : C \cap \Ga \to R$
satisfying $a(0)=1$.
\end{enumerate}
Set $C_{\mathbb{Z}}=C \cap \Ga$. We further assume that $\br(C_{\mathbb{Z}}) \subset \mathbb{N}$, that $\br$ and $\bd$ do not vanish simultaneously on $C_{\mathbb{Z}}$ and that $\mu(C_{\mathbb{Z}})$ is bounded below,
where $\mu=\bd / \br$. This implies in particular that $\bd(\alpha) >0$ if
$\br(\alpha)=0$ and $\alpha \in C_{\mathbb{Z}}$.

\vsp{.05in}

From this data we define an algebra
$$\mathbb{T}:=\bigoplus_{\alpha \in \Ga} Re^\alpha, \qquad e^{\alpha} \circ e^{\beta} = t^{\langle \alpha,\beta\rangle}e^{\alpha+\beta},$$
and the formal sum 
$$A=\sum_{\alpha \in C_{\mathbb{Z}}} a(\alpha) e^\alpha.$$

\begin{proposition}\label{P:HNrec}
There exists a unique factorisation
\begin{equation}\label{E:HNrec}
A=\prod_{\tau \downarrow}A_{\tau}, \qquad A_{\tau}= \left(\sum_{\mu(\alpha)=\tau} b_{\alpha}e^{\alpha}\right).
\end{equation}
\end{proposition}
\note{I suppose one needs completions for this theorem.
By our conventions from \S3.1, we
use product $b_{\al_1}\dots b_{\al_s}$ with $\mu(\al_1)>\dots>\mu(\al_s)$.
Therefore the proof should be formulated for the opposite algebra. I will comment it for now.
}

\vsp{.1in}

In our situation, assuming that $\deg D>2g-2$, we apply the above result to the following setting:
$R=\mathbb{Q}(q^{\frac{1}{2}})$, $t=-q^{\frac{1}{2}}$, $C=(\mathbb{N}^2)^I \subset \Ga=(\bZ^2)^I$,
$$\br(\al)=\sum_i r_i, \qquad \bd(\al)=\sum_i d_i,
\qquad\al=(\ur,\ud)\in\Ga=\bZ^I\oplus\bZ^I,
$$
the skew form on $\Ga$ is the one defined in Section~\ref{sec:invar}, and (see Corollary \ref{C:SDQS})
$$a_{\alpha}
=\sI^{\geq 0}_D(\alpha)
=\sI^{\ge0,\nil}_{K-D}(\al^*),\qquad
\al=(\ur,\ud)\in\Ga.$$
The latter invariants can be explicitly computed using Corollary \ref{cor:vol of strata} and results of \cite{schiffmann_indecomposable}.
By construction and formula \eqref{HN2}, the 
elements $b_{\alpha}$ uniquely determined by this data compute the volumes of the stacks of (positive) semistable quiver sheaves, i.e.
$$b_{\al}
=\sH^{\ge0}_D(\al)
=(-q^\oh)^{\hi_D(\al,\al)}\vol\rbr{\QS^{\geq 0,\sst}_D(\al)(\bk)}$$
and thus
$$\vol\rbr{\QS_D^\sst(\al)(\bk)}
=(-q^\oh)^{-\hi_D(\al,\al)}\lim_{N \to \infty} b_{(\ur,\ud+N\ur)}.$$

\providecommand{\bysame}{\leavevmode\hbox to3em{\hrulefill}\thinspace}
\providecommand{\href}[2]{#2}


\end{document}